\documentclass[journal,twoside,web]{./ieeecolor}
\usepackage{generic}
\usepackage{cite}
\usepackage{amsmath,amssymb,amsfonts}
\usepackage{algorithmic}
\usepackage{graphicx}

\usepackage{tikz}

\let\labelindent\relax
\usepackage{amsthm}

\newcommand\placeqed{\nobreak\hfill$\square$}

\newtheorem{theorem}{Theorem}
\newtheorem{corollary}{Corollary}
\newtheorem{lemma}{Lemma}

\newtheorem{proposition}{Proposition}

\theoremstyle{definition}
\newtheorem{definition}{Definition}
\newtheorem{problem}{Problem}
\newtheorem{assumption}{Assumption}

\theoremstyle{remark}
\newtheorem{remark}{Remark}

\usepackage{amsmath}
\usepackage{amssymb}

\usepackage[T1]{fontenc}
\usepackage{calligra}

\usepackage{stix}

\usepackage{dsfont}
\usepackage{hyperref}
\usepackage[inline]{enumitem}
\usepackage{graphicx}
\usepackage{caption}
\usepackage{subcaption}
\usepackage{tabularx}
\usepackage{booktabs}

\newcommand{\transp}{^\mathsf{T}}

\DeclareMathOperator*{\cart}{\scalerel*{\times}{\sum}}
\usepackage{scalerel}

\newenvironment{Blockmatrix}{
    \newcommand{\block}[6][draw]{
        \path[##1] (##3-1, -##2+1) rectangle (##5, -##4) node[pos=0.5] {$##6$};
    }
    \tikzpicture[
        x=14pt+width("C"), y=14pt+height("C"), 
        baseline={([yshift=-1ex]current bounding box.center)}
    ]
}{
    \endtikzpicture
}

\def\BibTeX{{\rm B\kern-.05em{\sc i\kern-.025em b}\kern-.08em
    T\kern-.1667em\lower.7ex\hbox{E}\kern-.125emX}}
\begin{document}
\title{Online Guaranteed Reachable Set Approximation for Systems with Changed Dynamics and Control Authority}
\author{Hamza El-Kebir, \IEEEmembership{Student Member, IEEE}, Ani Pirosmanishvili, Melkior Ornik, \IEEEmembership{Member, IEEE}
\thanks{Manuscript received September 25, 2021. }
\thanks{H. El-Kebir is with the Dept. of Aerospace Engineering at the University of Illinois Urbana-Champaign, Urbana, IL 61801 USA (e-mail: elkebir2@illinois.edu). }
\thanks{A. Pirosmanishvili is with the Dept. of Aerospace Engineering at the University of Illinois Urbana-Champaign, Urbana, IL 61801 USA (e-mail: anip2@illinois.edu). }
\thanks{M. Ornik is with the Dept. of Aerospace Engineering and the Coordinated Science Laboratory at the University of Illinois Urbana-Champaign, Urbana, IL 61801 USA (e-mail: mornik@illinois.edu).}}

\maketitle

\begin{abstract}
This work presents a method of efficiently computing inner and outer approximations of forward reachable sets for nonlinear control systems with changed dynamics and diminished control authority, given an a priori computed reachable set for the nominal system. The method functions by shrinking or inflating a precomputed reachable set based on prior knowledge of the system's trajectory deviation growth dynamics, depending on whether an inner approximation or outer approximation is desired. These dynamics determine an upper bound on the minimal deviation between two trajectories emanating from the same point that are generated on the nominal system using nominal control inputs, and by the impaired system based on the diminished set of control inputs, respectively. The dynamics depend on the given Hausdorff distance bound between the nominal set of admissible controls and the possibly unknown impaired space of admissible controls, as well as a bound on the rate change between the nominal and off-nominal dynamics. Because of its computational efficiency compared to direct computation of the off-nominal reachable set, this procedure can be applied to on-board fault-tolerant path planning and failure recovery. In addition, the proposed algorithm does not require convexity of the reachable sets unlike our previous work, thereby making it suitable for general use. We raise a number of implementational considerations for our algorithm, and we present three illustrative examples, namely an application to the heading dynamics of a ship, a lower triangular dynamical system, and a system of coupled linear subsystems.
\end{abstract}


\section{Introduction}

\IEEEPARstart{R}{eachability} analysis forms a fundamental part of dynamical system analysis and control theory, providing a means to assess the set of states that a system can reach under admissible control inputs at a certain point in time from a given set of initial states. Inner approximations of reachable sets are often used to attain a guaranteed estimate of the system's capabilities, while outer approximations can be used to verify that the system will not reach an unsafe state.

Outer approximations find widespread applications in fault-tolerance analysis and formal verification \cite{Blanke2006}, safe trajectory planning \cite{Vaskov2019}, and constrained feedback controller synthesis \cite{Coogan2020}. Methods for computing outer approximations of the reachable set include polynomial overapproximation \cite{Althoff2013}, direct set propagation \cite{Althoff2021}, and viscosity solutions to Hamilton--Jacobi--Bellman (HJB) equations \cite{Bansal2017}.

Inner approximations of reachable sets have received comparatively less attention than outer approximations \cite{Goubault2019}, but have recently seen use in path-planning problems with collision avoidance \cite{Schoels2020a}, as well as viability kernel computation \cite{Kaynama2012}, which can in turn be used for guaranteed trajectory planning \cite{Liniger2019}. Another application lies in safe set determination, in which one aims to obtain an inner approximation of the maximal robust control invariant set \cite{Gruber2021}. Methods for determining inner approximations of reachable sets have been based on various principles, including relying on polynomial inner approximation of the nonlinear system dynamics using interval calculus \cite{Goubault2017}, ellipsoid calculus \cite{Filippova2016}, and viscosity solutions to HJB equations \cite{Xue2020}. One major drawback of these methods is that they are computationally intensive and are often only suitable for systems of low dimension, making them ill-suited for online use.

In this work, we consider the problem of obtaining meaningful approximations of the reachable set of an off-nominal system by leveraging available a priori information on the nominal system dynamics. Here, we consider the reachable set of the nominal system, or an inner/outer approximation thereof, to be known prior to the system's operation. While obtaining reachable sets is often a computationally intensive task, it is often done during the design phase of a system, where computation times are less of a concern \cite{Lombaerts2013}. We then consider a change in dynamics of the system, for example due to partial system failure, which turns the nominal system into the \emph{off-nominal} system. Our goals is to obtain inner and outer approximations of the off-nominal reachable sets based on the nominal reachable set, in a way that can be applied in real-time.

In \cite{El-Kebir2021a}, the case in which the system experiences \emph{diminished control authority} was considered, i.e., its set of admissible control inputs has shrunk with respect to that of the nominal system. In addition, stringent restrictions on the family of system that could be considered were made in \cite{El-Kebir2021a}, in particular due to the requirement that the reachable set for the nominal \emph{and} off-nominal set be convex. This ultimately limits the applicability of the theory presented there to a limited set of problems.

Here, we do not impose any demands on the convexity of the reachable set, while still presenting an algorithm that can be run in real-time. This latter generalization to nonconvex sets requires a significant shift in the way we reason about the minimum deviation between trajectories of the nominal and off-nominal system. In this work, we also consider a change in the system dynamics, and provide methods for obtaining tighter inner \emph{and} outer approximations for the off-nominal reachable set with respect to what the theory of \cite{El-Kebir2021a} provides. This latter improvement follows from the fact that the theory in \cite{El-Kebir2021a} considers the trajectory deviation as expressed by the norm of the states, whereas here we  separately consider the deviation in single dimensions of the state.

To obtain inner and outer approximations of the off-nominal reachable set, we consider that an upper bound on the minimal rate of change of the trajectory deviation between the off-nominal system's trajectories with respect to those of the nominal system's is known, with both trajectories emanating from the same point. These growth dynamics provide an upper bound on the minimal rate of change between two trajectories emanating from the same point, with one trajectory being generated by the nominal set of control inputs, and the other by the off-nominal set of control inputs and the off-nominal dynamics. An upper bound on these growth dynamics can be obtained analytically during the design phase, and allows us to obtain an inner approximation to the off-nominal system's reachable set at low computational cost, in an online manner.

While other methods have been proposed to compute reachable sets under system impairment, due to their computational complexity, these have either used reduced order models, or have been limited to offline applications \cite{Norouzi2020}. In more extreme cases of system impairments, such as those were very little is known about the system's present capabilities, more conservative methods for computing reachable sets exist \cite{Shafa2021}. Here, we present a general algorithm that yields guaranteed inner and outer approximations, given limited knowledge of the failure modes as expressed by a bound on the trajectory deviation growth dynamics. We leverage the fact that the off-nominal system dynamics are related to the nominal system's dynamics, allowing us to leverage reachable sets computed for the nominal system, unlike in \cite{Shafa2021}. Given a sufficiently tight deviation growth bound, our approach can be applied online to high dimensional systems with no additional computational cost for the growth in system dimension.
 

The paper is organized as follows. First, we present preliminary theory in Section~\ref{sec:prelim}. Then, we present our main results Section~\ref{sec:main}, followed by a simulation example involving the heading dynamics of ship, as well as two general scalable system examples, in Section~\ref{sec:simulations}. Finally, we draw conclusions in Section~\ref{sec:conclusions}. In Appendix~\ref{app:generalizations}, we present a slightly more relaxed set of assumptions under which the theory presented continues to hold.

\section{Preliminaries}
\label{sec:prelim}

In the following, we denote by $\Vert \cdot \Vert$ the Euclidean norm. Given two sets $A, B \subseteq \mathbb{R}^n$, we denote by $A \oplus B$ Minkowski sum $\{a + b : a \in A, b \in B\}$. We denote a ball centered around the origin with radius $r > 0$ as $\mathcal{B}_r$. By $\mathcal{B}(x, r)$ we denote $\{x\} \oplus \mathcal{B}_r$. We denote by `$\cart$' the Cartesian product. We define $\mathbb{R}_+ := [0, \infty)$. We define the distance between two sets $A, B \subseteq \mathbb{R}^n$ to be
\begin{equation}\label{eq:distance between sets}
    d(A, B) := \sup_{a \in A} \inf_{b \in B} d(a, b),
\end{equation}
where $d$ is the Euclidean metric. We denote the Hausdorff distance as
\begin{equation}\label{eq:Hausdorff distance}
    d_{\mathrm{H}} (A, B) := \max\{d(A, B), d(B, A)\},    
\end{equation}
where $d$ is the Euclidean metric. An alternative characterization of the Hausdorff distance reads \cite[pp.~280--281]{Munkres2000}:
\begin{equation}\label{eq:Hausdorff characterization}
    d_{\mathrm{H}} (A, B) = \inf \{ \rho \geq 0 : A \subseteq B_{+\rho}, B \subseteq A_{+\rho} \},
\end{equation}
where $X_{+\rho}$ denotes the \emph{$\rho$-fattening} of $X$, i.e., $X_{+\rho} := \bigcup_{x \in X} \{ y \in \mathbb{R}^n : \Vert x - y \Vert \leq \rho \}$.

 Given a point $x \in S$ and a set $A \subseteq S$, we denote $d(x, A) := \inf_{y \in A} d(x, y)$. We denote by $\partial A$ the boundary of $A$ in the topology induced by the Euclidean norm. For a function $g : A \to B$, we denote by $g^{-1}$ the inverse of this function if an inverse exists, and by $\mathrm{dom}(g)$ the domain of the function (in this case $A$). We denote a \emph{multifunction} by $G : A \rightrightarrows B$, where $G$ maps elements of $A$ to subsets of $B$. Given a multifunction $G$, we define a \emph{differential inclusion} as being the set of ordinary differential equations $\dot{x} \in G(x)$ that have velocities in $G(x)$.
 
 Given two vectors $u, v \in \mathbb{R}^n$, we denote by $u \preceq v$ a component-wise nonstrict inequality, i.e., $u_i \leq v_i$ for $i = 1, \ldots, n$. By $|u|$, we denote a component-wise absolute value, such that $|u|_i = |u_i|$. Given a set $Y$, we denote its cardinality by $\# Y$. We use the abbreviation `a.e.' (almost every) to refer to statements that are true everywhere except potentially on some zero-measure sets. Given a real value $a \in \mathbb{R}$ we denotes its \emph{ceiling} by $\lceil a \rceil = \min ([a, \infty) \cap \mathbb{N})$. For a closed interval $[a, b] \subseteq \mathbb{R}$, we denote its length by $\mathrm{length}([a, b]) := b - a$. Given $N$ matrices $A^{(1)}, \ldots, A^{(N)}$, with $A^{(i)} \in \mathbb{R}^{n_i \times m_i}$ for each $i = 1, \ldots, n$, we define $\mathrm{diag}(\{ A^{(1)}, \ldots, A^{(N)} \})$ to be the block diagonal matrix formed by these matrices.
 

\subsection{Problem Statement}

Consider a dynamical system of the form
\begin{equation}\label{eq:nominal system dynamics}
\begin{split}
    \dot{x}(t) &= f(x(t), u(t)), \\
    x(0) &= x_0,
\end{split}
\end{equation}

where $t \geq 0$, $x \in \mathbb{R}^n$ is the state, and $u \in \mathscr{U} \subseteq \mathbb{R}^m$ is the control input, where $\mathscr{U}$ is some admissible set of control inputs. The dynamics have the form $f : \mathbb{R}^n \times \mathscr{U} \to \mathbb{R}^n$. We refer to these dynamics as the `nominal' dynamics.

We consider an impairment in the system dynamics, as well as the system's control authority, such that $\bar{u}(t) \in \bar{\mathscr{U}} \subseteq \mathscr{U} \subseteq \mathbb{R}^m$. The modified dynamics then read:
\begin{equation}\label{eq:off-nominal system dynamics}
\begin{split}
    \dot{\bar{x}}(t) &= g(\bar{x}(t), \bar{u}(t)), \\
    \bar{x}(0) &= x_0.
\end{split}
\end{equation}

We refer to these modified dynamics as the `off-nominal' dynamics.

\begin{definition}[Forward reachable set]
    We define a function $\phi : \mathbb{R}_+ \to \mathscr{U}$ as an \emph{admissible input signal}, if a unique solution to \eqref{eq:nominal system dynamics} exists given that input signal. The set of admissible control signals is defined as all possible admissible input signals $\mathbb{U} := \{\phi : \mathbb{R}_+ \to \mathscr{U}\}$.    
    
    We define a \emph{trajectory} $\varphi : \mathbb{R}_+ \times \mathbb{R}^n \times \mathbb{U} \to \mathbb{R}^n$ to be such that $x(t) = \varphi(t | x_0, \phi)$ satisfies \eqref{eq:nominal system dynamics} given initial state $x(0) = x_0 \in \mathbb{R}^n$ and input signal $u(t) = \phi(t) \in \mathbb{U}$, i.e.,
    \begin{equation*}
        \varphi(t | x_0, \phi) := x_0 + \int_0^t f(x(\tau), \phi(\tau)) \ \mathrm{d}\tau.
    \end{equation*}
    
    From the dynamics of \eqref{eq:nominal system dynamics}, we define a multifunction $F(t, x) := f(t, x, \mathscr{U}) : \mathbb{R}_+ \times \mathbb{R}^n \rightrightarrows \mathbb{R}^n$. This multifunction defines an \emph{ordinary differential inclusion} $\dot{x}(t) \in F(t, x(t))$, of which any instance of \eqref{eq:nominal system dynamics} is a part. We define the \emph{solution set} of this ordinary differential inclusion as follows:
    \begin{equation*}
        S_F (x_0) := \{ \varphi(\cdot | x_0, \phi) : x_0 \in \mathscr{X}_0, \phi \in \mathbb{U} \}.
    \end{equation*}

    Given a set of initial states $\mathscr{X}_0 \subseteq \mathbb{R}^n$, we define the \emph{forward reachable set} (FRS) at time $\tau \in \mathbb{R}_+$ as
    \begin{equation*}
    \begin{split}
        \mathbb{X}_\tau^\rightarrow (F, \mathscr{X}_0) &:= \bigcup_{x \in \mathscr{X}_0} \{x(\tau) : x \in S_F (x_0)\} \\
        &= \{\varphi(\tau | x_0, \phi) : x_0 \in \mathscr{X}_0, \phi \in \mathbb{U}\}. 
    \end{split}
    \end{equation*}
\end{definition}

We consider the following main problem, comprised of two parts: one relating to obtaining inner approximations of reachable sets, and the other concerned with obtaining outer approximations of reachable sets. In this work, we treat both the case of impaired control authority, as well as changed dynamics, simultaneously.

\begin{problem}[Off-nominal FRS approximation]
    Given the nominal dynamics $\dot{x}(t) = f(x(t), u(t)),$ the off-nominal dynamics $\dot{\bar{x}}(t) = g(\bar{x}(t), \bar{u}(t)),$ a set of admissible control inputs $\mathscr{U}$, (an inner (\emph{outer}) approximation of the) forward reachable set $\mathbb{X}_\tau^\rightarrow$ at time $\tau$, and the corresponding initial set of states $\mathscr{X}_f$, find an inner (\emph{outer}) approximation of the reachable set at time $\tau$, $\bar{\mathbb{X}}_\tau^\rightarrow$, for the dynamics $\dot{\bar{x}}(t) = g(\bar{x}(t), \bar{u}(t))$ and the admissible control inputs $\bar{\mathscr{U}} = h(\mathscr{U})$, for some control mapping $h : \mathscr{U} \to \bar{\mathscr{U}}$.
\end{problem}

As mentioned in the introduction, inner approximations of the off-nominal reachable set are useful for safety critical control, when guaranteed reachability is demanded. However, when dealing with collision avoidance, outer approximations of the off-nominal reachable set of a moving target are needed. This justifies the need for two separate approximation objectives.

\subsection{Generalized Nonlinear Trajectory Deviation Growth Bound}

As mentioned in the introduction, we wish to find an upper bound on the minimum normed distance between two trajectories emanating from the same, once governed by \eqref{eq:nominal system dynamics}, and the other by \eqref{eq:off-nominal dynamics}. We call this upper bound the \emph{trajectory deviation growth bound}. To this end, we first consider a means of obtaining and upper bound to the norm of the solution of a given ordinary differential equation (ODE). This particular ODE will be described by the rate of change of the deviation between two trajectories, which we refer to as the \emph{trajectory deviation growth dynamics}, as will be described shortly.




We consider the following general nonlinear time-varying dynamics
\begin{equation}\label{eq:time-varying dynamics}
    \dot{x}(t) = h(t, x(t), u(t)), \quad x(0) = x_0 \in \mathbb{R}^n,
\end{equation}

Our goal is to find an upper bound for the magnitude of $x(t)$, given particular assumptions on the form of control input $u$ and the function $h$, over a finite period of time. We make the following assumption on the growth rate of $x$:
\begin{assumption}\label{assumption:nonlinear growth condition}
    For all $x \in \mathbb{R}^n$, $u \in \mathscr{U}$, $t_0 \leq t < \infty$
    \begin{equation*}
        \Vert h(t, x, u) \Vert \leq a(t) w(\Vert x \Vert, \Vert u \Vert) + b(t),
    \end{equation*}
    where $a, b$ are continuous and positive and $w$ is continuous, monotonic, nondecreasing and positive. In addition, $w$ is uniformly monotonically nondecreasing in $\Vert u \Vert$.
\end{assumption}

\begin{theorem}[Extended Bihari inequality {\cite[Theorem~3.1]{El-Kebir2021a}}]\label{thm:generalized Bihari inequality for controlled dynamical systems}
    Let $x(t)$ be a solution to the equation
    \begin{equation*}
        \dot{x} = h(t, x, u), \quad 0 \leq t_0 \leq t < \infty,
    \end{equation*}
    where $h(t, x, u) : [t_0, \infty) \times \mathbb{R}^n \times \mathscr{U} \to \mathbb{R}^n$ is continuous for $t_0 \leq t < \infty$, and $\mathscr{U} \subseteq \mathbb{R}^m$ is compact and satisfies $\max_{u \in \mathscr{U}} \Vert u \Vert = \delta$. Let Assumption~\ref{assumption:nonlinear growth condition} hold. Then,
    \begin{equation}\label{eq:Bihari inequality}
        \Vert x(t) \Vert \leq G^{-1} \left[ G\left( \Vert x(t_0) \Vert + \int_{t_0}^t b(\tau) \mathrm{d}\tau \right) + \int_{t_0}^t a(\tau) \mathrm{d}\tau \right],
    \end{equation}
    where the expression on the right-hand side is strictly increasing in $t$. In \eqref{eq:Bihari inequality}, we define
    \begin{equation*}\label{eq:big G integral}
        G(r) := \int_{r_0}^r \frac{\mathrm{d}s}{w(s, \delta)}, \quad r > 0, r_0 > 0,
    \end{equation*}
    for arbitrary $r_0 > 0$ and for all $t \geq t_0$ for which it holds that
    \begin{equation*}
        G\left( \Vert x(t_0) \Vert + \int_{t_0}^t b(\tau) \mathrm{d}\tau \right) + \int_{t_0}^t a(\tau) \mathrm{d}\tau \in \mathrm{dom}(G^{-1}).
    \end{equation*}
\end{theorem}

\begin{remark}
    Theorem~\ref{thm:generalized Bihari inequality for controlled dynamical systems} is a generalization of the Gronwall--Bellman inequality. It can be reduced to the Gronwall-Bellman inequality by taking $w(u) = u$ and $G(u) = \log u$ (see \cite[Remark 2.3.2, p.~109]{Pachpatte1998} for a more in-depth discussion).
\end{remark}

\begin{corollary}\label{corollary:generalized boundedness of trajectory deviation}
    For any $x_0 \in \mathscr{X}_0 \subseteq \mathbb{R}^n$, where $\mathscr{X}_0$ is compact, and any initial time $t_{\mathrm{init}} \in [t_0, \infty)$ and final time $t_f \in [t_{\mathrm{init}}, \infty)$, consider a trajectory $x(t)$ satisfying $x(t_{\mathrm{init}}) = x_0$ and $\dot{x}(t) = f(t, x(t), u(t))$, with $u(t) \in \mathbb{U}$. Consider a trajectory $\bar{x}(t)$ with $\bar{x}(t_{\mathrm{init}}) = x_0$ and $\dot{\bar{x}}(t) = f(t, \bar{x}(t), \bar{u}(t))$, such that $\bar{u}(t) \in \bar{\mathbb{U}}$ satisfies $\sup_{t \in [t_{\mathrm{init}}, t_f]} \Vert u(t) - \bar{u}(t) \Vert \leq \epsilon$. Let $\tilde{f} (t) := f(t, x(t), u(t)) - f(t, \bar{x}(t), \bar{u}(t))$, $\tilde{x}(t) := x(t) - \bar{x}(t)$, and $\tilde{u}(t) := u(t) - \bar{u}(t)$. Let the following bound hold for $t \in [t_{\mathrm{init}}, t_f]$, and for \emph{any} $x(t)$ and $\bar{x}(t)$ satisfying the previous hypotheses: $\Vert \tilde{f}(t) \Vert \leq \tilde{a}(t) \tilde{w}(\Vert \tilde{x}(t) \Vert, \Vert \tilde{u}(t) \Vert) + \tilde{b}(t)$ with $\tilde{a}, \tilde{b}, \tilde{w}$ satisfying the assumptions given in Assumption~\ref{assumption:nonlinear growth condition}. Then, $\tilde{x}(t)$ satisfies
        \begin{equation*}\label{eq:Bihari inequality deviation}
            \Vert \tilde{x}(t) \Vert \leq G^{-1} \left[ G\left( \int_{t_0}^t \tilde{b}(\tau) \mathrm{d}\tau \right) + \int_{t_0}^t \tilde{a}(\tau) \mathrm{d}\tau \right] =: \eta(t, \epsilon)
        \end{equation*}
        for all $t \in [t_{\mathrm{init}}, t_f]$.
\end{corollary}

\begin{proof}
    Given the premise, this claim follows directly from Theorem~\ref{thm:generalized Bihari inequality for controlled dynamical systems}.
\placeqed\end{proof}

\begin{corollary}\label{corollary:generalized boundedness of trajectory deviation with changed initial conditions}
    For any $x_0 \in \mathscr{X}_0 \subseteq \mathbb{R}^n$, where $\mathscr{X}_0$ is compact, and any initial time $t_{\mathrm{init}} \in [t_0, \infty)$ and final time $t_f \in [t_{\mathrm{init}}, \infty)$, consider a trajectory $x(t)$ satisfying $x(t_{\mathrm{init}}) = x_0$ and $\dot{x}(t) = f(t, x(t), u(t))$, with $u(t) \in \mathbb{U}$. Consider a trajectory $\bar{x}(t)$ with $\bar{x}(t_{\mathrm{init}}) = \bar{x}_0$, where $\bar{x}_0 \in \bar{\mathscr{X}}_0 \subseteq \mathbb{R}^n$, and $\dot{\bar{x}}(t) = f(t, \bar{x}(t), \bar{u}(t))$, such that $\bar{u}(t) \in \bar{\mathbb{U}}$ satisfies $\sup_{t \in [t_{\mathrm{init}}, t_f]} \Vert u(t) - \bar{u}(t) \Vert \leq \epsilon$. Let $\tilde{f} (t) := f(t, x(t), u(t)) - f(t, \bar{x}(t), \bar{u}(t))$, $\tilde{x}(t) := x(t) - \bar{x}(t)$, and $\tilde{u}(t) := u(t) - \bar{u}(t)$. Let $d_{\text{H}} (\mathscr{X}_0, \bar{\mathscr{X}}_0) \leq \kappa$. Let the following bound hold for $t \in [t_{\mathrm{init}}, t_f]$, and for \emph{any} $x(t)$ and $\bar{x}(t)$ satisfying the previous hypotheses: $\Vert \tilde{f}(t) \Vert \leq \tilde{a}(t) \tilde{w}(\Vert \tilde{x}(t) \Vert, \Vert \tilde{u}(t) \Vert) + \tilde{b}(t)$ with $\tilde{a}, \tilde{b}, \tilde{w}$ satisfying the assumptions given in Assumption~\ref{assumption:nonlinear growth condition}. Then, $\tilde{x}(t)$ satisfies
        \begin{equation*}\label{eq:Bihari inequality deviation}
        \begin{split}
            \Vert \tilde{x}(t) \Vert &\leq G^{-1} \left[ G\left( \kappa + \int_{t_0}^t \tilde{b}(\tau) \mathrm{d}\tau \right) + \int_{t_0}^t \tilde{a}(\tau) \mathrm{d}\tau \right] \\
            &=: \eta(t, \epsilon, \kappa)
        \end{split}
        \end{equation*}
        for all $t \in [t_{\mathrm{init}}, t_f]$.
\end{corollary}

\begin{proof}
    Similarly to Corollary~\ref{corollary:generalized boundedness of trajectory deviation}, given the premise, this claim follows directly from Theorem~\ref{thm:generalized Bihari inequality for controlled dynamical systems}.
\placeqed\end{proof}

\subsubsection{Generalization to off-nominal dynamics}

We now consider the following nonlinear time-varying off-nominal dynamics:
\begin{equation}\label{eq:off-nominal dynamics}
    \dot{x}(t) = g(t, x(t), u(t)), \quad x(0) = x_0 \in \mathbb{R}^n,
\end{equation}

which gives rise to the following assumption that relates these unknown dynamics to the known nominal system dynamics:
\begin{assumption}\label{assumption:off-nominal dynamics norm bound}
    For all $x \in \mathbb{R}^n$, $u \in \mathscr{U}$, $t_0 \leq t < \infty$, we have
    \begin{equation*}
        \Vert g(t, x, u) - f(t, x, u) \Vert \leq \gamma(t),
    \end{equation*}
    where $\gamma$ is a positive, continuous function on $[t_0, \infty)$.
\end{assumption}

This assumption gives rise to the following lemma
\begin{lemma}\label{lm:impaired dynamics norm bound}
    Let Assumption~\ref{assumption:off-nominal dynamics norm bound} hold true. Consider functions $\bar{a}, \bar{b}, \bar{w}$, in the sense of Assumption~\ref{assumption:nonlinear growth condition}, which apply to the following dynamics
     \begin{equation*}
        \dot{\tilde{x}}(t) = f(t, x(t) + \tilde{x}(t), u(t) + \tilde{u}(t)) - f(t, x(t), u(t)), \quad \tilde{x}(t_0) = 0.
    \end{equation*}
    We consider nominal control signal $u \in \mathbb{U}_\epsilon$, and an off-nominal control signal of the form $(u + \bar{u}) \in \mathbb{U}$ are such that $\sup_{\bar{u}(t), t \in [t_0, \infty)} \Vert \bar{u}(t) \Vert \leq \epsilon$.
    
    In addition, consider the following off-nominal trajectory deviation dynamics:
    \begin{equation*}
        \dot{\hat{x}}(t) = g(t, x(t) + \hat{x}(t), u(t) + \hat{u}(t)) - f(t, x(t), u(t)), \quad \hat{x}(t_0) = 0.
    \end{equation*}
    Control signals $u$ and $\hat{u}$ are defined similarly to those of the nominal dynamics. We have a nominal control signal $u \in \mathbb{U}_\epsilon$, and an off-nominal control signal of the form $(u + \hat{u}) \in \mathbb{U}$ such that $\sup_{\hat{u}(t), t \in [t_0, \infty)} \Vert \hat{u}(t) \Vert \leq \epsilon$.
    
    We define $\hat{b}(t) := \tilde{b}(t) + \gamma(t)$. We have that $\Vert \tilde{x}(t) \Vert \leq \eta(t, \epsilon)$, where $\eta(t, \epsilon)$ is obtained as in Corollary~\ref{corollary:generalized boundedness of trajectory deviation} with $\tilde{a} = \bar{a}$, $\tilde{b} = \hat{b}$, $\tilde{w} = \bar{w}$.
\end{lemma}

\begin{proof}
    This result follows straightforwardly by application of the the triangle inequality on the trajectory deviation growth dynamics and Corollary~\ref{corollary:generalized boundedness of trajectory deviation}.
\placeqed\end{proof}

\section{Main results}
\label{sec:main}

We now present the main results of this paper. We give a method for inner and outer approximation of the forward reachable sets for off-nominal systems that are subject to both diminishment of control authority and changed dynamics.

Unlike in \cite{El-Kebir2021a}, a new hyperrectangular version of the Bihari inequality is required to obtain inner and outer approximations to more generalized reachable sets. As will be specified later, the only requirements imposed on these reachable sets will be that they are nonempty, compact, and connected. To construct a hyperrectangular Bihari inequality, we present a modified nonlinear bound on the deviation dynamics:

\begin{definition}[$(a_i, b_i, w_i)$-bounded growth]\label{def:hyperrectangular nonlinear growth bound}
    We say that a function $h : [t_0, \infty) \times \mathbb{R}^n \times \mathscr{U} \to \mathbb{R}^n$ has \emph{$(a_i, b_i, w_i)$-bounded growth} if for all $x \in \mathbb{R}^n$, $u \in \mathscr{U}$, $t_0 \leq t < \infty$, the following inequality holds:
    \begin{equation*}
        | h_i (t, x, u) | \leq a_i (t) w_i (| x_i |, \Vert u \Vert) + b_i (t),
    \end{equation*}
    for each $i \in \{1, \ldots, n\}$,
    where $a_i, b_i$ are continuous and positive and $w_i$ is continuous, monotonic, nondecreasing and positive in both of its arguments.
\end{definition}


Given this definition, we can now formulate a generalization to Corollary~\ref{corollary:generalized boundedness of trajectory deviation}:

\begin{lemma}\label{lemma:hyperrectangular generalized boundedness of trajectory deviation}
    For any $x_0 \in \mathscr{X}_0 \subseteq \mathbb{R}^n$, where $\mathscr{X}_0$ is compact, and any initial time $t_{\mathrm{init}} \in [t_0, \infty)$ and final time $t_f \in [t_{\mathrm{init}}, \infty)$, consider a trajectory $x(t)$ satisfying $x(t_{\mathrm{init}}) = x_0$ and $\dot{x}(t) = f(t, x(t), u(t))$, with $u(t) \in \mathbb{U}$. Consider a trajectory $\bar{x}(t)$ with $\bar{x}(t_{\mathrm{init}}) = x_0$ and $\dot{\bar{x}}(t) = f(t, \bar{x}(t), \bar{u}(t))$, such that $\bar{u}(t) \in \bar{\mathbb{U}}$ satisfies $\sup_{t \in [t_{\mathrm{init}}, t_f]} \Vert u(t) - \bar{u}(t) \Vert \leq \epsilon$. Let $\tilde{f} (t) := f(t, x(t), u(t)) - f(t, \bar{x}(t), \bar{u}(t))$, $\tilde{x}(t) := x(t) - \bar{x}(t)$, and $\tilde{u}(t) := u(t) - \bar{u}(t)$. Let $\tilde{f}$ be of $(\tilde{a}_i, \tilde{b}_i, \tilde{w}_i)$-bounded growth for $t \in [t_{\mathrm{init}}, t_f]$. Then, $\tilde{x}(t)$ satisfies
        \begin{equation}\label{eq:Bihari inequality deviation componentwise}
            | \tilde{x}_i (t) | \leq G_i^{-1} \left[ G_i\left( \int_{t_0}^t \tilde{b}_i (\tau) \mathrm{d}\tau \right) + \int_{t_0}^t \tilde{a}_i (\tau) \mathrm{d}\tau \right] =: \eta_i (t, \epsilon)
        \end{equation}
        for each $i \in \{1, \ldots, n\}$ and for all $t \in [t_{\mathrm{init}}, t_f]$, where the $G_i$ are defined as $G$ in Theorem~\ref{thm:generalized Bihari inequality for controlled dynamical systems}.
\end{lemma}

\begin{proof}
    The proof follows by repeated application of Corollary~\ref{corollary:generalized boundedness of trajectory deviation} on each dimension.
\placeqed\end{proof}

In what follows, we will equivalently express the bound \eqref{eq:Bihari inequality deviation componentwise} as
    \begin{equation*}\label{eq:Bihari inequality deviation hyperrectangular}
            |\tilde{x} (t)| \preceq \eta(t, \epsilon),
        \end{equation*}
        where $\eta(t, \epsilon) := [\begin{smallmatrix} \eta_1 (t, \epsilon) & \cdots & \eta_n (t, \epsilon) \end{smallmatrix}]\transp$.

Whereas Corollary~\ref{corollary:generalized boundedness of trajectory deviation} gives a bound on the trajectory deviation as a ball, Lemma~\ref{lemma:hyperrectangular generalized boundedness of trajectory deviation} provides a hyperrectangular trajectory deviation bound. This distinction is key in the theorem that will follow next.

We first introduce two new definitions relating to the hyperrectangular trajectory deviation growth bound.

\begin{definition}[Hyperrectangular fattening]\label{def:hyperrectangular fattening}
    Given a compact set $X \subseteq \mathbb{R}^n$, a hyperrectangular fattening by $\rho \in \mathbb{R}^n$ with $\rho \succeq 0$ is defined as:
    \begin{equation*}
        X_{\boxplus \rho} := \bigcup_{x \in X} \left[ \{x\} \oplus \left( \cart_{i =  1}^n [-\rho_i, \rho_i] \right) \right].
    \end{equation*}
\end{definition}

\begin{figure}[h]
\centerline{\includegraphics[width=\linewidth]{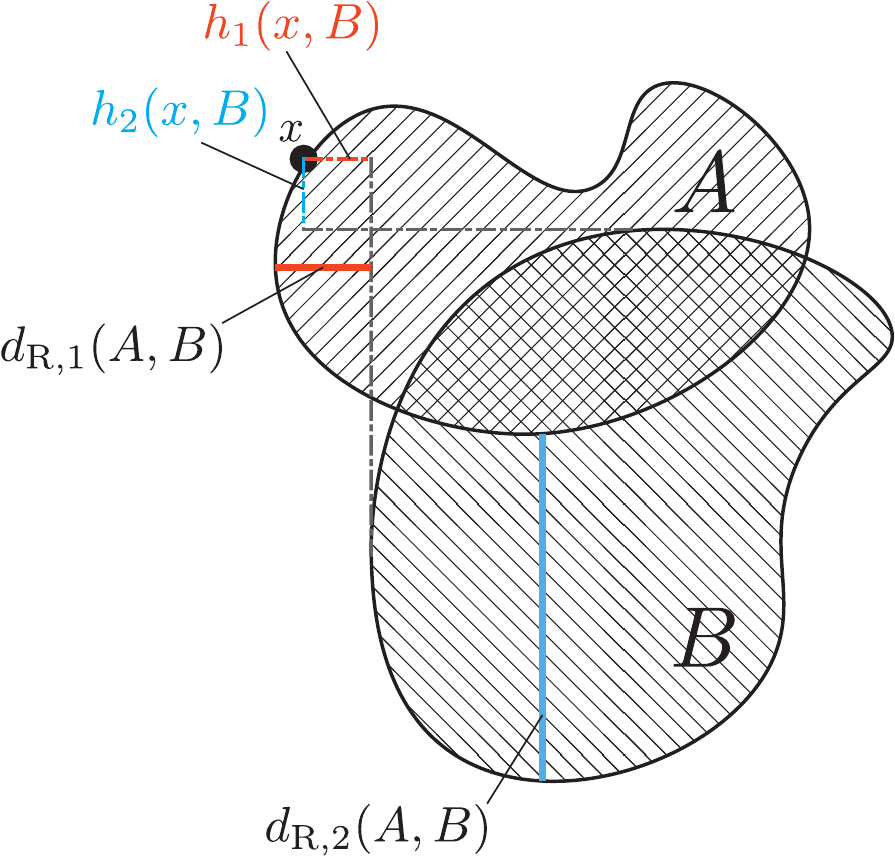}}
\caption{Illustration of the hyperrectangular distance between two compact sets, as determined by distances to intersecting hyperplanes. Distances $h_i (x, B)$ are shown for each axis, as well as $d_{\mathrm{R}, i} (A, B)$ to show that $d_{\mathrm{R}, i} (A, B) \geq h_i (x, B)$.}
\label{fig:hyperrectangular distance illustration}
\end{figure}

\begin{definition}[Hyperrectangular distance]\label{def:hyperrectangular distance}
    Given two compact sets $A, B \subseteq \mathbb{R}^n$, we denote by $d_{\mathrm{R}} (A, B)$ their \emph{hyperrectangular distance}, defined as follows.
    
    On each Cartesian axis $i = 1, \ldots, n$, consider each point $x \in A$ and $y \in B$. We denote by $h_i (x, B)$ the following:
    \begin{equation*}
    \begin{split}
        h_i (x, B) = \min \{ |x_i - y_i| : y \in B\}.
    \end{split}
    \end{equation*}
    If $B = \emptyset$, we say $h_i (x, B) = \infty$.
    
    We denote the hyperrectangular distance as $d_{\mathrm{R}, i} (A, B) = \max \{ \max_{x \in A} h_i (x, B), \max_{y \in B} h_i (y, A) \}$. Combining each component, we define
    \begin{equation*}
        d_{\mathrm{R}} (A, B) := \begin{bmatrix} d_{\mathrm{R}, 1} (A, B) & \cdots & d_{\mathrm{R}, n} (A, B) \end{bmatrix}\transp.
    \end{equation*}
\end{definition}

\begin{remark}
    The value of $h_i (x, B)$ shown above corresponds to the shortest distance from $x$ where there exists a hyperplane on a line from $x$ in the direction of $e_i$, with a normal direction $e_i$, that intersects $B$ (see Fig.~\ref{fig:hyperrectangular distance illustration} for an illustration). If such an intersecting hyperplane does not exist, we say $h_i(x, B) = \infty$. We then have  $d_{\mathrm{R}, i} (A, B) = \max \{ \max_{x \in A} h_i (x, B), \max_{y \in B} h_i (y, A) \}$, as shown in Fig.~\ref{fig:hyperrectangular distance illustration}.

    From Definition~\ref{def:hyperrectangular distance} it trivially follows that for two compact sets $A, B \subseteq \mathbb{R}^n$, we have $A_{\boxplus \rho} \supseteq B$ and $B_{\boxplus \rho} \supseteq A$ if and only if $\rho \succeq d_{\text{R}}(A, B)$. This is analogous to the fattening-based characterization of the Hausdorff distance in \eqref{eq:Hausdorff characterization}.
\end{remark}

Before we can proceed, we must impose a number of mild conditions on the differential inclusion defined by dynamics \eqref{eq:nominal system dynamics} and \eqref{eq:off-nominal system dynamics}, as well as the initial set of states $\mathscr{X}_0$. In particular, we wish to show that the reachable set $\mathbb{X}_t^\rightarrow (F, \mathscr{X}_0)$ is \textit{connected}. This property is key in proving the main result of this work.

We first present a prerequisite lemma on the connectedness of the solution set of a differential inclusion. To this end, we require the definition of the following metric space, as well as two propositions that provide sufficient conditions for $F$ to produce solution sets with connected and compact values.

\begin{definition}\label{def:Staicu function space}
    Let $S$ be a function space defined as
    \begin{equation*}
        S := \left\{x \in C([0, \infty), \mathbb{R}^n) : \dot{x} \in L_{\mathrm{loc}}^1 ([0, \infty), \mathbb{R}^n)\right\},
    \end{equation*}
    endowed with distance metric $d_S$ defined as
    \begin{equation}
        d_S (x, y) := \Vert x(0) - y(0) \Vert + \sum_{k = 1}^\infty \frac{1}{2^k} \frac{\int_0^k \Vert \dot{x}(t) - \dot{y}(t) \Vert \ \mathrm{d}t}{1 + \int_0^k \Vert \dot{x}(t) - \dot{y}(t) \Vert \ \mathrm{d}t}.
    \end{equation}
\end{definition}

In Definition~\ref{def:Staicu function space}, $(S, d_S)$ is a complete metric space \cite[Prop.~1, p.~1007]{Staicu2004}.

\begin{lemma}[{\cite[Thm.~1, p.~1010]{Staicu2004}}]\label{lemma:path-connectedness of solution set}
    Consider a differential inclusion
    \begin{equation*}
    \begin{split}
        \dot{x}(t) &\in F(t, x(t)), \quad \text{a.e.} \ t \in [0,T], \\
        x(0) &= x_0 \in \mathbb{R}^n,
    \end{split}
    \end{equation*}
    where $F : [0, \infty) \times \mathbb{R}^n \rightrightarrows \mathbb{R}^n$ is a compact-valued multifunction.
    Let $F(t, x)$ be path-connected for all $(t, x) \in \mathbb{R}^n \times \mathbb{R}^n$, and let $F(t, x)$ be continuous and Lipschitz in $t$ and $x$.
    Then for any $x_0 \in \mathbb{R}^n$, the set of solutions
    \begin{equation*}
    \begin{split}
        S_F(x_0) := \{x \in S : &x(0) = x_0,  \\
        &\dot{x}(t) \in F(t, x(t)) \ \text{a.e.} \ t \in [0, \infty)\},
    \end{split}
    \end{equation*}
    is path-connected in the space $(S, d_S)$.
\end{lemma}

\begin{remark}
    Since the conditions of Lemma~\ref{lemma:path-connectedness of solution set} on $F$ will be required throughout this work, we will look at the applicability of these conditions to commonly encountered classes of dynamical systems. We list two here:
    \begin{enumerate}
        \item \emph{Affine-in-control systems}: Consider functions $g : [0, \infty) \times \mathbb{R}^n \to \mathbb{R}^n$ and $h : [0, \infty) \times \mathbb{R}^n \to \mathbb{R}^{n \times m}$ that form the following differential equation:
        \begin{equation*}
            \dot{x}(t) = g (t, x(t)) + h (t, x(t)) u(t).
        \end{equation*}
        From this system, we can introduce a differential inclusion defined by a multifunction
        \begin{equation*}
            F(t, x) := \{g (t, x)\} \oplus h (t, x) \mathscr{U},
        \end{equation*}
        where $\mathscr{U} \subseteq \mathbb{R}^m$ is nonempty, compact, and path-connected. If $g$ and $h$ are continuous and Lipschitz in their arguments, then $F$ will satisfy conditions of Lemma~\ref{lemma:path-connectedness of solution set}.
        \item \emph{General nonlinear systems}: For a dynamical system of the form of \eqref{eq:nominal system dynamics}, a sufficient condition for $F(t, x) := f(t, x, \mathscr{U})$ to satisfy the condition of Lemma~\ref{lemma:path-connectedness of solution set}, is for $\mathscr{U}$ to be nonempty, compact, and path-connected, and $f(t, x)$ to be continuous and Lipschitz in $t$ and $x$.
    \end{enumerate}
    Relaxations to the conditions of Lemma~\ref{lemma:path-connectedness of solution set} are presented in Appendix~\ref{app:generalizations}.
\end{remark}

We proceed to prove that path-connectedness of $S_F (x_0)$ implies that the values of $S_F^t (x_0) := \{x(t) : x \in S_F (x_0)\}$ are connected for all $t \in [0, \infty)$.

\begin{proposition}[Path-connectedness of $S_F^t (x_0)$]\label{prop:connectedness of solution set values}
    For some multifunction $F$ satisfying the hypotheses of Lemma~\ref{lemma:path-connectedness of solution set}, and some $x_0 \in \mathbb{R}^n$, the set $S_F^t (x_0) = \{x(t) : x \in S_F (x_0)\}$ is path-connected for all $t \in [0, \infty)$.
\end{proposition}

\begin{proof}
    We consider the case where $\# S_F (x_0) > 1$; the case of a singleton is trivial. To show that the values of $S_F (x_0)$ are path-connected, let us first note that for any $x, y \in S_F (x_0)$, there exists a continuous path $p : [0, 1] \to S_F (x_0)$ such that $p(0) = x$ and $p(1) = y$. For all $t$, for any $a, b \in S_F (x_0) (t)$, there exist at least one pair of functions $(x_a, x_b) \subseteq S_F (x_0)$ such that $x_a (t) = a$, $x_b (t) = b$. Let $p_{a, b} : [0, 1] \to S_F (x_0)$ be a path that connects $x_a$ and $x_b$, which exists by path-connectedness of $S_F (x_0)$. Then, $p(\cdot)(t) : [0, 1] \to \mathbb{R}^n$ is a path between $a$ and $b$. Hence, the values of $S_F (x_0)(t)$ are path-connected for all $t \in [0, \infty)$.
\placeqed\end{proof}

In Proposition~\ref{prop:connectedness of solution set values}, we have shown that the sets $S_F^t (x_0)$ for $t \in [0, \infty)$ are path-connected. In our main theorem, we will also require that these sets are compact and nonempty. The following proposition guarantees this.

\begin{proposition}[$S_F^t (x_0)$ forms a path-connected continuum]\label{prop:solution set values continuum}
    For a differential inclusion
    \begin{equation*}
    \begin{split}
        \dot{x}(t) &\in F(t, x(t)), \quad \text{a.e.} \ t \in [0,T], \\
        x(0) &= x_0 \in \mathbb{R}^n,
    \end{split}
    \end{equation*}
    where $F : [0, \infty) \times \mathbb{R}^n \rightrightarrows \mathbb{R}^n$ is a compact-valued multifunction. Let $F$ satisfy the hypotheses of Lemma~\ref{lemma:path-connectedness of solution set}.

    Then, for any $x_0 \in \mathbb{R}^n$, the solution set $S_F (x_0)$ is a path-connected \emph{continuum} in $S$, i.e., a nonempty, compact, path-connected subset of $S$. In addition, the sets $S_F^t (x_0)$ are path-connected continua in $\mathbb{R}^n$ for all $t \in [0, \infty)$.
\end{proposition}

\begin{proof}
    The fact that $S_F (x_0)$ is a nonempty compact subset of $S$ follows from \cite[Thm.~5.1, p.~228]{Tolstonogov1986}. It is clear that the values of $S_F (x_0)$ will be nonempty, since $\# S_F (x_0) > 0$. Given some $t \in [0, \infty)$, we define a functional $E_t : S \to \mathbb{R}^n$ as $E_t (x) := x(t)$. If $E_t$ can be shown to be continuous, then $E_t$ is a \emph{preserving} map in the sense of \cite[p.~21]{Gerlits2004}, i.e., the image $E_t (S_F(x_0)) = S_F^t (x_0) \subseteq \mathbb{R}^n$ preserves the compactness and path-connectedness properties of $S_F(x_0)$. We proceed to show that $E_t$ is indeed a continuous linear functional.
    
    We require linearity to show that $E_t$ is uniformly continuous on all of $S$. To show that $E_t$ is linear, consider two functionals $x, y \in S$, and a scalar $\alpha \in \mathbb{R}$. We clearly find
    \begin{equation*}
    \begin{split}
        E_t (x + y) &= (x + y)(t) = x(t) + y(t) = E_t (x) + E_t (y), \\
        E_t (\alpha x) &= (\alpha x)(t) = \alpha x(t) = \alpha E_t (x).
    \end{split}
    \end{equation*}
    
    If $E_t$ is linear, by \cite[Thm.~1.1, p.~54]{Taylor1986} it suffices to show that $E_t$ is continuous at any one $y \in S$. Continuity of the functional is shown next by using the $\delta$--$\epsilon$ criterion \cite[p.~52]{Taylor1986}.
    
    We proceed to show that for each $\epsilon > 0$, there exists some $\epsilon' > 0$ such that $\Vert E_t (x) - E_t (y) \Vert < \epsilon$ for all $x \in S$ such that $d_S (x, y) < \epsilon'$. To this end, let us define $a_k := \int_0^k \Vert \dot{x}(t) - \dot{y}(t) \Vert \ \mathrm{d}t$. By definition of the solution set, we have for any $x, y \in S_F (x_0)$:
    \begin{equation*}
    \begin{split}
        d_S (x, y) &= \Vert x(0) - y(0) \Vert + \sum_{k = 1}^\infty \frac{1}{2^k} \frac{a_k}{1 + a_k} \\
        &= \sum_{k = 1}^\infty \frac{1}{2^k} \frac{a_k}{1 + a_k} = \sum_{k = 1}^\infty \frac{1}{2^k} \frac{1}{1 + a_k^{-1}},        
    \end{split}
    \end{equation*}
    since $x(0) = y(0) = x_0$.
    
    We know by path-connectedness of $S_F (x_0)$ that for any $x \in S_F (x_0)$, given some $\epsilon' > 0$, there exists $y \in S_F (x_0) \setminus \{x\}$ such that $d_S (x, y) < \epsilon'$. In general, we can upper-bound the difference between values of $x$ and $y$ at time $t$ as follows:
    \begin{equation*}
    \begin{split}
        \Vert x(t) - y(t) \Vert &= \left\Vert x(0) + \int_0^t \dot{x} (s) \ \mathrm{d}s - y(0) - \int_0^t \dot{y} (s) \ \mathrm{d}s \right\Vert \\
        &= \left\Vert \int_0^t \dot{x} (s) - \dot{y} (s) \ \mathrm{d}s \right\Vert \leq \int_0^t \Vert \dot{x}(s) - \dot{y}(s) \Vert \ \mathrm{d}s,
    \end{split}
    \end{equation*}
    which follows from Jensen's inequality \cite[p.~109]{Folland1999}.
    
    Given some $\epsilon > 0$ and $t \in (0, \infty)$, let $K := \lceil t \rceil \in \mathbb{N}$. Since $\int_0^t \Vert \dot{x}(s) - \dot{y}(s) \Vert \ \mathrm{d}s < \epsilon$ implies $d(x(t), y(t)) < \epsilon$, it suffices to show that there exists some $\epsilon' > 0$, such that any $y \in S_F (x_0) \setminus \{x\}$ that satisfies $d_S (x, y) < \epsilon'$ yields $\int_0^t \Vert \dot{x}(s) - \dot{y}(s) \Vert \ \mathrm{d}s < \epsilon$. We choose $\epsilon'$ as follows:
    \begin{equation*}
    \begin{split}
        d_S (x, y) &= \sum_{k = 1}^\infty \frac{1}{2^k} \frac{1}{1 + a_k^{-1}} \\
        &\geq \left( \sum_{k = 1}^K \frac{1}{2^k} \frac{1}{1 + \epsilon^{-1}} \right) + \left( \sum_{k = K+1}^\infty \frac{1}{2^k} \frac{1}{1 + a_k^{-1}} \right) \\
        &\geq \sum_{k = 1}^K \frac{1}{2^k} \frac{1}{1 + \epsilon^{-1}} =: \epsilon'(t, \epsilon).
    \end{split}
    \end{equation*}
    We may evaluate $\epsilon'(t, \epsilon)$ as
    \begin{equation*}
        \epsilon'(t, \epsilon) = \frac{(1 - 2^{-K}) \epsilon}{1 + \epsilon} > 0.
    \end{equation*}
    If we consider $y \in S_F (x_0) \setminus \{x\}$ such that $d_S (x, y) < \epsilon'(t, \epsilon)$, we have therefore shown that we obtain $d(x(t), y(t)) < \epsilon$; at least one such $y$ exists by path-connectedness of $S_F (x_0)$. We thus proved continuity of $E_t$.
    
    Having shown that $E_t$ is a continuous (linear) operator (and therefore a preserving map), we have proven that $E_t (S_F(x_0))$ is a path-connected continuum for all $t \in [0, \infty)$ and $x_0 \in \mathbb{R}^n$.
\placeqed\end{proof}


Given the result of Proposition~\ref{prop:solution set values continuum}, we can now show that given the above conditions and a condition on the initial set of states, the reachable set $\mathbb{X}_t (F, \mathscr{X}_0)$ is also path-connected.

\begin{lemma}\label{lemma:connectedness of reachable set}
    For a differential inclusion
    \begin{equation*}
    \begin{split}
        \dot{x}(t) &\in F(t, x(t)), \quad \text{a.e.} \ t \in [0,T], \\
        x(0) &= x_0 \in \mathbb{R}^n,
    \end{split}
    \end{equation*}
    where $F : [0, \infty) \times \mathbb{R}^n \rightrightarrows \mathbb{R}^n$ satisfies all conditions listed in Proposition~\ref{prop:solution set values continuum}, given a path-connected continuum $\mathscr{X}_0 \subseteq \mathbb{R}^n$, reachable set $\mathbb{X}_t^\rightarrow (F, \mathscr{X}_0)$ is path-connected for all $t \in [0, \infty)$.
\end{lemma}

\begin{proof}
    We draw upon \cite[Cor.~4.5, p.~233]{Zhu1991}, which says that the choice of $F$ in Lemma~\ref{lemma:path-connectedness of solution set} is sufficient for the solution set $S_F : \mathbb{R}^n \rightrightarrows S$ to be continuous on $\mathbb{R}^n$. In other words, the solution set $S_F$ has a continuous dependence on the initial value.
    
    We characterize the reachable set as follows:
    
    \begin{equation*}
        R_F (t) := \mathbb{X}_t^\rightarrow (F, \mathscr{X}_0) = \bigcup_{x_0 \in \mathscr{X}_0} S_F^t (x_0),
    \end{equation*}
    
    It is clear that for any two values $a, b \in R_F (t)$, there exist $x_0, x_0' \in \mathscr{X}_0$ such that $a \in S_F^t (x_0)$ and $b \in S_F^t (x_0')$. Since $\mathscr{X}_0$ is path-connected and $S_F$ is continuous, there exists a continuous path $p_0 : [0, 1] \to \mathscr{X}_0$ connecting $x_0$ and $x_0'$. Therefore, the solution sets for $S_F (x_0)$ and $S_F (x_0')$ are connected by a path $p_s = S_F \ \circ \ p_0 : [0, 1] \rightrightarrows S$. Since $S_F' (x_0, x_0') := \bigcup_{\tau \in [0, 1]} p_s (\tau)$ is path-connected, this implies that its values are also path-connected, analogous to the latter part of the proof of Proposition~\ref{prop:connectedness of solution set values}. Hence, $R_F (t)$ is path-connected for all $t \in [0, \infty)$.
\placeqed\end{proof}

We can now provide a means of inner and outer approximating the off-nominal reachable set based on a hyperrectangular trajectory deviation growth bound.

\begin{figure}[h]
\centerline{\includegraphics[width=0.85\linewidth]{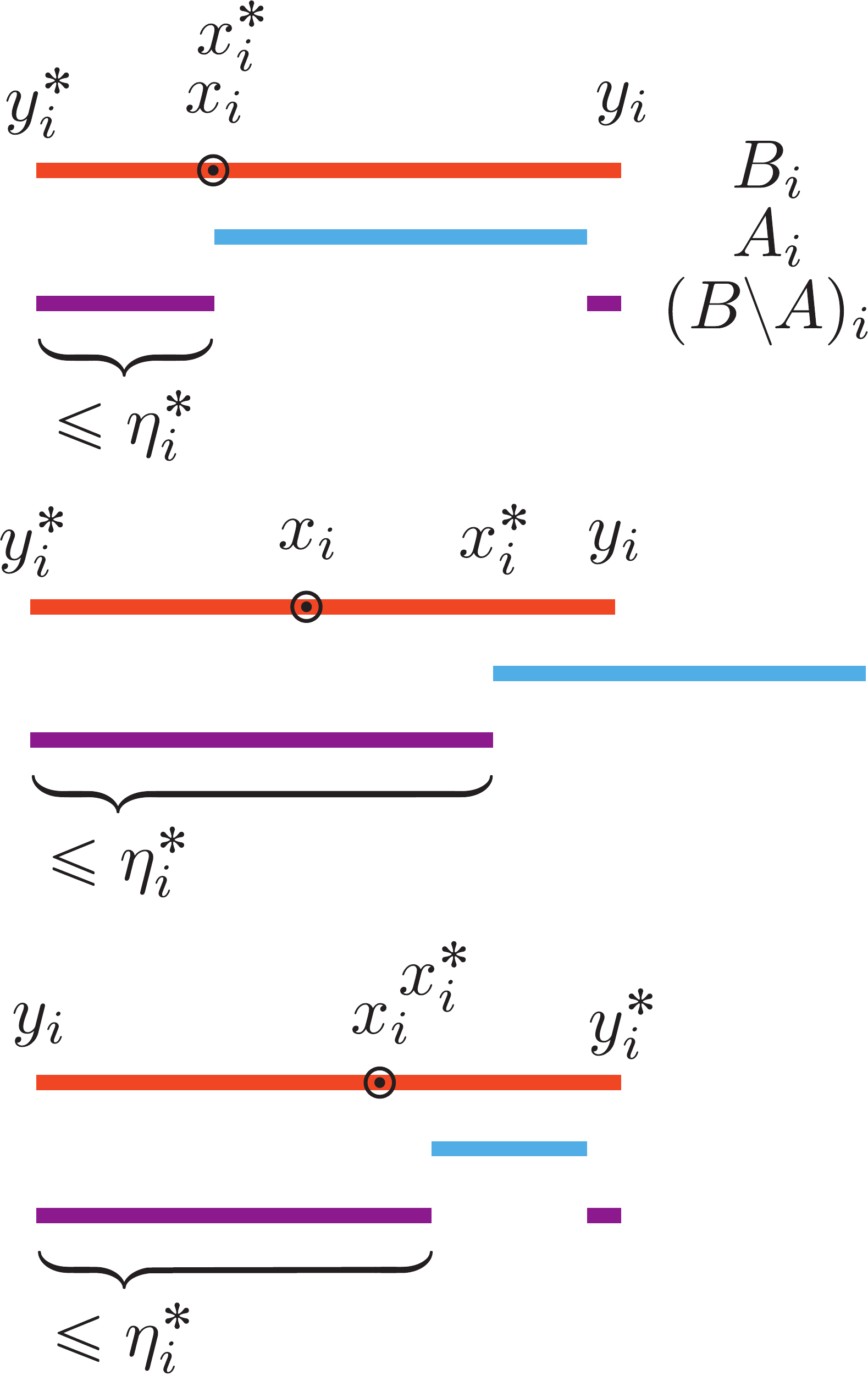}}
\caption{Illustration of some of the arrangements considered in producing the various contradictions in the proof of Theorem~\ref{thm:reachable set guaranteed underapproximation impaired general}(iii).}
\label{fig:hyperrectangular distance contradictions}
\end{figure}

\begin{theorem}[General FRS inner approximation with changed dynamics]\label{thm:reachable set guaranteed underapproximation impaired general}
    Let $f : [0, \infty) \times \mathbb{R}^n \times \mathscr{U} \to \mathbb{R}^n$, $g : [0, \infty) \times \mathbb{R}^n \times \mathscr{V} \to \mathbb{R}^n$, where $\mathscr{U}, \mathscr{V} \subseteq \mathbb{R}^m$ are such that $d_{\mathrm{R}} (\mathscr{U}, \mathscr{V}) \preceq \epsilon$. Let $G (t, x) = g(t, x, \mathscr{V})$ and $F (t, x) = f(t, x, \mathscr{U})$, and let $\mathscr{X}_0 \subseteq \mathbb{R}^n$, the set of initial states, and initial time $t_0 \in \mathbb{R}_+$ be given. Let $\mathscr{X}_0$, $f, g$, and $\mathscr{U}, \mathscr{V}$ satisfy the conditions of Lemma~\ref{lemma:connectedness of reachable set}. Let the hypotheses of Lemma~\ref{lemma:hyperrectangular generalized boundedness of trajectory deviation} be satisfied with $\bar{\mathscr{U}} = \mathscr{V}$, $t_{\mathrm{init}} = t_0$, and $t_f > t_{\mathrm{init}}$. Let $\eta(t, \epsilon)$ be obtained as in Lemma~\ref{lemma:hyperrectangular generalized boundedness of trajectory deviation}. Then:
    \begin{enumerate}[label=(\roman*), wide, labelwidth=!, labelindent=0pt]
        \item For each $x_0 \in \mathscr{X}_0$ there exists a trajectory $x(t)$ emanating from $x (t_0) = x_0$ with $\dot{x}(t) \in F(t, x(t))$ and a trajectory $y (t)$ satisfying $y (t_0) = x_0$ and $\dot{y} (t) \in G (t, y(t))$ such that $|x(t) - y(t)| \preceq \eta (t, \epsilon)$ for all $t \in [t_0, t_f]$;
        \item Let $T \in [0, t_f - t_0]$, and let $\eta^* = \eta(t_0 + T, \epsilon)$. For all $t \in [t_0, t_0 + T]$, $d_{\mathrm{R}} [\mathbb{X}^\rightarrow_t (G, \mathscr{X}_0), \mathbb{X}^\rightarrow_t (F, \mathscr{X}_0)] \preceq \eta^*$;
        \item For all $t \in [t_0, t_0 + T]$, \begin{equation*}
            \mathbb{X}^\rightarrow_t (F, \mathscr{X}_0) \setminus (\partial \mathbb{X}^\rightarrow_t (F, \mathscr{X}_0))_{\boxplus \eta^*} \subseteq \mathbb{X}^\rightarrow_t (G, \mathscr{X}_0).
        \end{equation*}
    \end{enumerate}
\end{theorem}

\begin{proof}
    
    
    \begin{enumerate}[label=(\roman*), wide, labelwidth=!, labelindent=0pt]
    \item This fact follows directly from Lemma~\ref{lemma:hyperrectangular generalized boundedness of trajectory deviation}.
    \item From (i), the maximal distance between two points in $\mathbb{X}_t^\rightarrow (F, \mathscr{X}_0)$ and $\mathbb{X}_t^\rightarrow (G, \mathscr{X}_0)$ is upper-bounded by $\eta(t, \epsilon)$. In Theorem~\ref{thm:generalized Bihari inequality for controlled dynamical systems} it was shown that $\eta(t, \epsilon)$ is increasing in $t$, meaning that the hyperrectangular distance bound holds for all times $t \leq t_0 + T$.
    
    \item  We define $A = \mathbb{X}^\rightarrow_t (G, \mathscr{X}_0)$ and $B = \mathbb{X}^\rightarrow_t (F, \mathscr{X}_0)$ for any $t \in [t_0, t_0 + T]$. Note that in this proof, unlike in \cite{El-Kebir2021a}, $A$ and $B$ need not be convex, making the proof more technically challenging. We wish to show that $C := B \setminus (\partial B)_{\boxplus \eta^*}$ is a subset of $A$.
    
    The inclusion $C \subseteq A$ can equivalently be shown by demonstrating that for all $x \in B \setminus A$, we have $x \in (\partial B)_{\boxplus \eta^*}$; we prove the latter claim by contradiction. Assume that there indeed exists $x \in B \setminus A$ such that $x \not\in (\partial B)_{\boxplus \eta^*}$. This point then will have some $i \in \{1, \ldots, n\}$, such that have $d_i (x, \partial B) := \min \{|x_i - y_i| : y \in \partial B\} > \eta_i^*$. From the characterization of $\eta^*$ in Lemma~\ref{lemma:hyperrectangular generalized boundedness of trajectory deviation} and the definition of the hyperrectangular distance in Definition~\ref{def:hyperrectangular distance}, we have $d_{\text{R}}(A, B) \leq \eta^*$. In light of the contradiction given above, this inequality would then necessarily yield the following equivalent contradiction:
    \begin{equation}\label{eq:hyperrectangular distance contradiction}
    \begin{split}
        d_i (x, \partial B) > \max\{ &\max_{x' \in A} \min \{|x_i' - y_i| : y \in B\}, \\
        &\max_{x' \in B} \min \{|x_i' - y_i| : y \in A\} \},
    \end{split}
    \end{equation}
    which implies
    \begin{equation*}
        \min \{ |x_i - y_i| : y \in \partial B\} > \max_{x' \in \partial B} \min \{ |x_i' - y_i| : y \in A \}.
    \end{equation*}
    
    Let $y^* \in \partial B$ be such that $d_i (x, y^*) = \min \{ |x_i - y_i| : y \in \partial B \}$. By the hypotheses, in particular the compactness of $B$, there exists some $x^* \in A$ such that $|x^* - y^*| \preceq \eta^*$.
    
%
    
    We identify two cases: a) $x \in \partial B \cap (B \setminus A)$, and b) $x \in B \setminus (A \cup \partial B)$. In case a), we find $d_i (x, \partial B) = 0$, which produces the desired contradiction. 
    
    We now consider case b). Let us denote by $X_i$ the projection of all points of a set $X \subseteq \mathbb{R}^n$ onto the $i$-th Cartesian axis, such that $X_i \subseteq \mathbb{R}$. Since $A$ and $B$ are both compact, connected, and nonempty, we find that $A_i$ and $B_i$ are closed intervals in $\mathbb{R}$ for each $i = 1, \ldots, n$. This fact follows trivially by considering that the projection operation is continuous, and continuous maps are \emph{preserving} maps in the sense of \cite[p.~21]{Gerlits2004}, i.e., their images preserve connectedness and compactness. From \cite[Thm.~12.8, p.~116]{Sutherland2009}, any connected subspace of $\mathbb{R}$ is an interval, which shows that the $A_i, B_i$ are compact (or closed) intervals.
    
    We can then note that for any $y_i^*$, there exists some $x^* \in A$ such that $|x_i^* - y_i^*| \leq \eta^*_i$ by Lemma~\ref{lemma:hyperrectangular generalized boundedness of trajectory deviation}.
    Finally, let $y_i \in \partial B_i$ be the other point in the boundary of the interval $B_i$ such that $y_i \neq y_i^*$.
    
    We can identify twelve arrangements of $(x, x^*, y, y^*)$, barring cases of symmetry. Some of these arrangements are inadmissible, as shown below. In what follows, we must have $y^* \neq y$, since $x$ would otherwise be on the boundary of $B$, which was treated as case a). Also, necessarily, $x \neq x^*$, since we have $x \not\in A$. We prove that for each admissible arrangement, the claim of \eqref{eq:hyperrectangular distance contradiction} does not hold; an illustration of some of these arrangements in given in Fig.~\ref{fig:hyperrectangular distance contradictions}. For some $a, b, c, d \in \mathbb{R}$, we will denote the ordering $a < b < c < d$ by the shorthand notation $a, b, c, d$. We have:
    \begin{enumerate}
        \item $x_i^*, x_i, y_i^*, y_i$: Not admissible since $x \not\in B \setminus A$.
        \item $x_i^*, x_i, y_i, y_i^*$; Not admissible since $x \not\in B \setminus A$, and inconsistent with the definition of $y^*$.
        \item $x_i^*, y_i^*, x_i, y_i$: In this case, by Lemma~\ref{lemma:hyperrectangular generalized boundedness of trajectory deviation} we have $d(x_i^*, y_i^*) \leq \eta^*_i$. Since we have $d(x_i, y_i^*) < d(x_i^*, y_i^*)$, we find that $d(x_i, y_i^*) < \eta^*_i$.
        \item $x_i^*, y_i^*, y_i, x_i$: Not admissible since $x \not\in B \setminus A$.
        \item $x_i^*, y_i, x_i, y_i^*$: We have $d(x_i^*, y_i^*) \leq \eta^*_i$ and $d(x_i, y_i^*) < d(x_i^*, y_i^*)$, which implies $d(x_i, y_i^*) < \eta^*_i$.
        \item $x_i^*, y_i, y_i^*, x_i$: Not admissible since $x \not\in B \setminus A$.
        \item $x_i, x_i^*, y_i^*, y_i$: Not admissible since $x \not\in B \setminus A$.
        \item $x_i, x_i^*, y_i, y_i^*$: Not admissible since $x \not\in B \setminus A$, and not consistent with the definition of $y^*$.
        \item $x_i, y_i^*, x_i^*, y_i$: Not admissible since $x \not\in B \setminus A$.
        \item $x_i, y_i, x_i^*, y_i^*$: Not admissible since $x \not\in B \setminus A$, and inconsistent with the definition of $y^*$.
        \item $y_i^*, x_i^*, x_i, y_i$: We have $d(x_i, y_i) \geq d(x_i, y_i^*)$, $d(x_i^*, y_i) > d(x_i, y_i)$, as well as $d(x_i^*, y_i) \leq \eta^*_i$. From this it follows that $d(x_i, y_i^*) < \eta^*_i$.
        \item $y_i^*, x_i, x_i^*, y_i$: We have $d(x_i, y_i^*) < d(x_i^*, y_i^*) \leq \eta^*_i$.
    \end{enumerate}
    
    Having considered all cases, in all admissible scenarios it follows that the statement in \eqref{eq:hyperrectangular distance contradiction} is false. This in turn contradicts the claim that there exists $x \in B \setminus A$ such that $x \not\in (\partial B)_{\boxplus \eta^*}$. Therefore, we have proven that $C \subseteq A$. 

    \end{enumerate}\vspace{-1.25pc}\placeqed\end{proof}
    
We now present two corollaries that cover the case of a changed set of initial conditions, as well as guaranteed overapproximations of the reachable set.
    
\begin{corollary}\label{cor:reachable set guaranteed innerapproximation impaired general changed initial conditions}
    Let the hypotheses of Theorem~\ref{thm:reachable set guaranteed underapproximation impaired general} hold. In addition, let there be an off-nominal set of initial conditions $\bar{\mathscr{X}}_0$ that is nonempty, closed, and connected, such that $d_{\text{H}} (\mathscr{X}_0, \bar{\mathscr{X}}_0) \leq \kappa$. Define $\eta(t, \epsilon, \kappa)$ as in Corollary~\ref{corollary:generalized boundedness of trajectory deviation with changed initial conditions}. Then:
    \begin{enumerate}[label=(\roman*), wide, labelwidth=!, labelindent=0pt]
        \item For each $x_0 \in \mathscr{X}_0$ there exists a trajectory $x(t)$ emanating from $x (t_0) = x_0$ with $\dot{x}(t) \in F(t, x(t))$ and a trajectory $y (t)$ satisfying $y (t_0) = x_0'$, for some $x_0' \in \{x_0\}_{+\kappa} \cap \bar{\mathscr{X}}_0$, and $\dot{y} (t) \in G (t, y(t))$ such that $|x(t) - y(t)| \preceq \eta (t, \epsilon, \kappa)$ for all $t \in [t_0, t_f]$;
        \item Let $T \in [0, t_f - t_0]$, and let $\eta^{**} = \eta(t_0 + T, \epsilon, \kappa)$. For all $t \in [t_0, t_0 + T]$, $d_{\mathrm{R}} [\mathbb{X}^\rightarrow_t (G, \bar{\mathscr{X}}_0), \mathbb{X}^\rightarrow_t (F, \mathscr{X}_0)] \preceq \eta^{**}$;
        \item For all $t \in [t_0, t_0 + T]$, \begin{equation*}
            \mathbb{X}^\rightarrow_t (F, \mathscr{X}_0) \setminus (\partial \mathbb{X}^\rightarrow_t (F, \mathscr{X}_0))_{\boxplus \eta^{**}} \subseteq \mathbb{X}^\rightarrow_t (G, \bar{\mathscr{X}}_0).
        \end{equation*}
    \end{enumerate}
\end{corollary}

\begin{proof}
    The proof immediately follows by consideration of Corollary~\ref{corollary:generalized boundedness of trajectory deviation with changed initial conditions}. 
\placeqed\end{proof}
    
\begin{corollary}\label{cor:reachable set guaranteed overapproximation impaired general}
    Let the hypotheses of Theorem~\ref{thm:reachable set guaranteed underapproximation impaired general} and Corollary~\ref{cor:reachable set guaranteed innerapproximation impaired general changed initial conditions} hold. Then:
    \begin{enumerate}[label=(\roman*)]
        \item For each $x_0 \in \mathscr{X}_0$ there exists a trajectory $x(t)$ emanating from $x (t_0) = x_0$ with $\dot{x}(t) \in F(t, x(t))$, and a trajectory $y (t)$ satisfying $y (t_0) = y_0$ with $y_0 \in \bar{\mathscr{X}}_0$ such that $\Vert x_0 - y_0 \Vert \leq \kappa$ and $\dot{y} (t) \in G (t, y (t))$, such that $|x(t) - y (t) | \preceq \eta(t, 2\delta + \epsilon, \kappa)$ for all $t \in [t_0, t_f]$ and $i = 1, \ldots, n$, where $\delta := \max_{u \in \mathscr{U}} \Vert u \Vert$;
        \item For all $t \in [t_0, t_f]$, $d_{\mathrm{R}} [\mathbb{X}^\rightarrow_t (G, \bar{\mathscr{X}}_0), \mathbb{X}^\rightarrow_t (F, \mathscr{X}_0)] \preceq \eta (t_f, 2\delta + \epsilon, \kappa)$;
        \item For all $t \in [t_0, t_f]$, \begin{equation*}
            \mathbb{X}^\rightarrow_t (G, \bar{\mathscr{X}}_0) \subseteq (\mathbb{X}^\rightarrow_t (F, \mathscr{X}_0))_{\boxplus \eta(t_f, 2\delta + \epsilon, \kappa)}.
        \end{equation*}
    \end{enumerate}
\end{corollary}

\begin{proof}\begin{enumerate}[label=(\roman*), wide, labelwidth=!, labelindent=0pt]
    \item This claim follows from Lemma~\ref{lm:impaired dynamics norm bound}, where we have used the following inequality:
    \begin{equation*}
        \Vert u(t) - v(t) \Vert \leq \Vert u(t) \Vert + \Vert v(t) \Vert \leq \delta + \delta + d_{\mathrm{H}} (\mathscr{U}, \mathscr{V}) \leq 2\delta + \epsilon,
    \end{equation*}
    which follows from the triangle inequality, as well as the definition of the Hausdorff distance.
    \item This fact follows directly from Lemma~\ref{lm:impaired dynamics norm bound}.
    \item The proof here is similar to that of Theorem~\ref{thm:reachable set guaranteed underapproximation impaired general}(iii), where we consider that any point in $A$ has a counterpart in $B$ that is distance $ \eta(t_f, 2\delta + \epsilon)$ away. The proof is immediate from this consideration and Lemma~\ref{lemma:hyperrectangular generalized boundedness of trajectory deviation}.
\end{enumerate}
\vspace{-1.25pc}\placeqed\end{proof}

\begin{remark}
    In Corollaries~\ref{cor:reachable set guaranteed innerapproximation impaired general changed initial conditions} and \ref{cor:reachable set guaranteed overapproximation impaired general}, the Hausdorff distance upper bound $\kappa$ on the initial set of states does not decrease the quality of the inner approximation with increasing time, similar to the quantity $2\delta + \epsilon$. In fact, the approximations decrease in tightness with increasing time solely on account of the `looseness' of the functions $\tilde{a}, \tilde{b}, \tilde{w}$ of Lemma~\ref{lemma:hyperrectangular generalized boundedness of trajectory deviation} that upper-bound the trajectory deviation growth.
\end{remark}
    
%

\section{Simulation results}
\label{sec:simulations}

We consider three numerical examples: a simplified representation of the heading dynamics of a sea-faring vessel, and lower triangular dynamical system, and an interconnected system of linear subsystems. The restriction to lower dimension systems stems from computational limitations in obtaining the nominal reachable sets with sufficient accuracy, as well as a desire to keep derivations concise. We will show how Theorem~\ref{thm:reachable set guaranteed underapproximation impaired general} and Corollary~\ref{cor:reachable set guaranteed overapproximation impaired general} can be applied to these systems. For both examples, we have used the CORA MATLAB toolkit \cite{Althoff2018} to compute the nominal and off-nominal reachable sets for illustrative purposes; in reality, such tools are not required to apply the theory presented here. In practice, the nominal reachable set would be computed prior to the system's operation using a similar toolkit. The methods used in such toolkits can often not be used online because of hardware limitations and poor scalability, hence the need for an approach such as ours.

In practice, it is difficult to obtain a hyperrectangular slimming of the form $X \setminus (\partial X)_{\boxplus \rho}$ using widely used software packages. For this reason, we propose an alternative using a conservative ball-based slimming operation. It is obvious that the following holds:
\begin{equation}\label{eq:ball-based slimming approximation}
    X \setminus (\partial X)_{+ \Vert \rho \Vert} \subseteq X \setminus (\partial X)_{\boxplus \rho},
\end{equation}
where $\Vert \rho \Vert$ denotes the Euclidean norm of the vector $\rho$. This follows from the fact that the ball $\mathcal{B}_{+ \Vert \rho \Vert}$ includes the hyperrectangle $\cart_{i =  1}^n [-\rho_i, \rho_i]$. In the following, we will show approximations based on naive ball-based slimming using single elements $\rho$, which give an indication of the shape of a true hyperrectangular slimming in that particular dimension. We also give a guaranteed inner approximation by applying a ball-based slimming operation with radius $\Vert \rho \Vert$.

Compared to \cite{El-Kebir2021a}, this approximation approach yields tighter approximations, since the bounds obtained there are greater or equal to $\Vert \rho \Vert$, as a bound on the Euclidean norm of the trajectory deviation is used there.

%

\subsection{Norrbin's Ship Steering Dynamics}

We first consider Norrbin's model of the heading dynamics of a ship sailing at constant velocity \cite{Fossen1992}:
\begin{equation*}
    \dot{x}(t) = \begin{bmatrix} \dot{x}_1 (t) \\ \dot{x}_2 (t) \end{bmatrix} = f(x(t), u(t)) = \begin{bmatrix}
        x_2 \\
        -\frac{v}{2 l} (x_2 + x_2^3)
    \end{bmatrix}
    +
    \begin{bmatrix}
        0 \\
        \frac{v^2}{2 l^2}
    \end{bmatrix}
    u,
\end{equation*}
where $x_1$ is the heading (or yaw) angle, and $x_2$ is the heading rate; in this example, $u$ denotes the rudder angle, $v$ denotes the fixed cruise speed, and $l$ denotes the vessel length. As can be observed in the dynamics, a vessel's ability to make turns is strongly correlated with its velocity (higher speeds provide greater resistance, but induce stronger rudder authority), as well as the length of the vessel. Intuitively, a longer vessel is harder to turn due to its inertia and hydrodynamic resistance of the hull. The dynamics are of second order, as a rudder deflection naturally induces a yaw moment.

We can find the following bound on trajectory deviation growth:
\begin{equation}\label{eq:norrbin deviation growth bound}
\begin{split}
    &|\tilde{f}(\bar{x}, \bar{u})| = |f (x + \tilde{x}, u + \tilde{u}) - f(x, u)| \\
    &\preceq \begin{bmatrix}
        |\tilde{x}_2| \\
        \frac{v}{2 l} (|\tilde{x}_2| + |\tilde{x}_2|^3 + 3 |\tilde{x}_2|^2 |M_2| + 3 |\tilde{x}_2| |M_2|^2) + \frac{v^2}{2 l^2} |\tilde{u}|
     \end{bmatrix},
\end{split}
\end{equation}
where $M_2 = \max_{y \in \mathbb{X}_{t}^\rightarrow (F, \mathscr{X}_0)} |y_2|$. $M_2$ can be determined since the nominal reachable set $\mathbb{X}_{t}^\rightarrow (F, \mathscr{X}_0)$ is available to us. We note that \eqref{eq:norrbin deviation growth bound} contains an integrator in state $\tilde{x}_1$, which allows us to obtain a hyperrectangular trajectory deviation bound as follows. We first compute the deviation bound on state $\tilde{x}_2$, such that $|\tilde{x}_2(t)| \leq \eta_2 (t)$. We then compute an upper bound on $\tilde{x}_1$, which is of the form $\eta_1 (t) = \int_{t_0}^t \eta_2 (\tau) \ \text{d}\tau$, which gives $|\tilde{x}_1(t)| \leq \eta_1 (t)$. Alternatively, a more conservative expression for $\eta_1$ can be obtained by defining it as $\eta_1 (t) := \eta(t_f) t$, which follows from the fact that $\eta$ is strictly increasing (see the proof of Theorem~\ref{thm:generalized Bihari inequality for controlled dynamical systems}).

\subsubsection{Diminished Control Authority}

In this example, we define multifunction $F$ as $F(x) := f(x, \mathscr{U})$, with $\mathscr{U} = [-25^\circ, 25^\circ]$ and the impaired control set is $\bar{\mathscr{U}} = [-20^\circ, 20^\circ]$, hence $\epsilon = d_{\mathrm{H}} (\mathscr{U}, \bar{\mathscr{U}}) = 5^\circ$. We consider the initial set of states to be a singleton: $\mathscr{X}_0 = \{[\begin{smallmatrix} 0^\circ & 5^\circ/\text{s} \end{smallmatrix}]\transp\}$. The nominal velocity is taken as $v = 5$ m/s, and the length of the vessel is $l = 45$ m.

We first consider the case of diminished control authority, i.e., the case in which the system dynamics remain the same, but the control inputs are draw from $\bar{\mathscr{U}}$ instead of $\mathscr{U}$.

We evaluate the reachable set at $t = 0.5$ s, $t = 1$ s, and $t = 3$ s, yielding the results shown in Fig.~\ref{fig:norrbin diminished}. We have given a guaranteed inner approximation based on the conservative ball-based slimming approach, as well as guaranteed intervals in each Cartesian axis using the entries of $\eta(t)$. These guaranteed intervals are shown as cross-hatched areas, and give an indication of what a hyperrectangular slimming would have produced, in addition to providing a guarantee that there exists at least one state in the off-nominal reachable set that has one of its coordinates on one of the intervals. Unlike the results in \cite{El-Kebir2021a}, the quality of the inner approximations degrades little with time (see Fig.~\ref{fig:norrbin diminished 3000}). This feature can be attributed to the fact that we are using a hyperrectangular growth bound in this work, as opposed to a more conservative norm-based bound.

An application to computing a guaranteed reachable set of the positions of the ship after control authority diminishment based on Norrbin's model has been prepared as a video\footnote{Demonstration of the computation of a guaranteed reachable set for the Norrbin ship model under diminished rudder authority: \url{https://youtu.be/5eUINOGJ_0Y}}.

\begin{figure*}[!t]
\centering
\begin{subfigure}[t]{0.2\linewidth}
    \centering
    \includegraphics[height=0.23\paperheight]{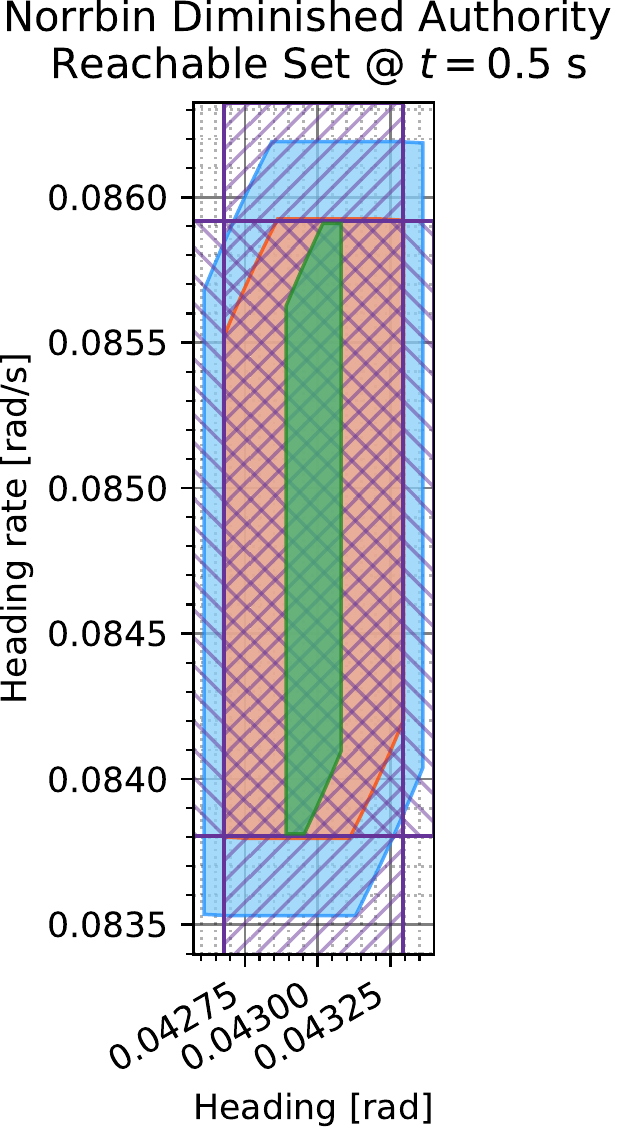}
    \caption{Reachable sets at 0.5 s}
    \label{fig:norrbin diminished 500}
\end{subfigure}
\hfill
\begin{subfigure}[t]{0.2\linewidth}
    \centering
    \includegraphics[height=0.23\paperheight]{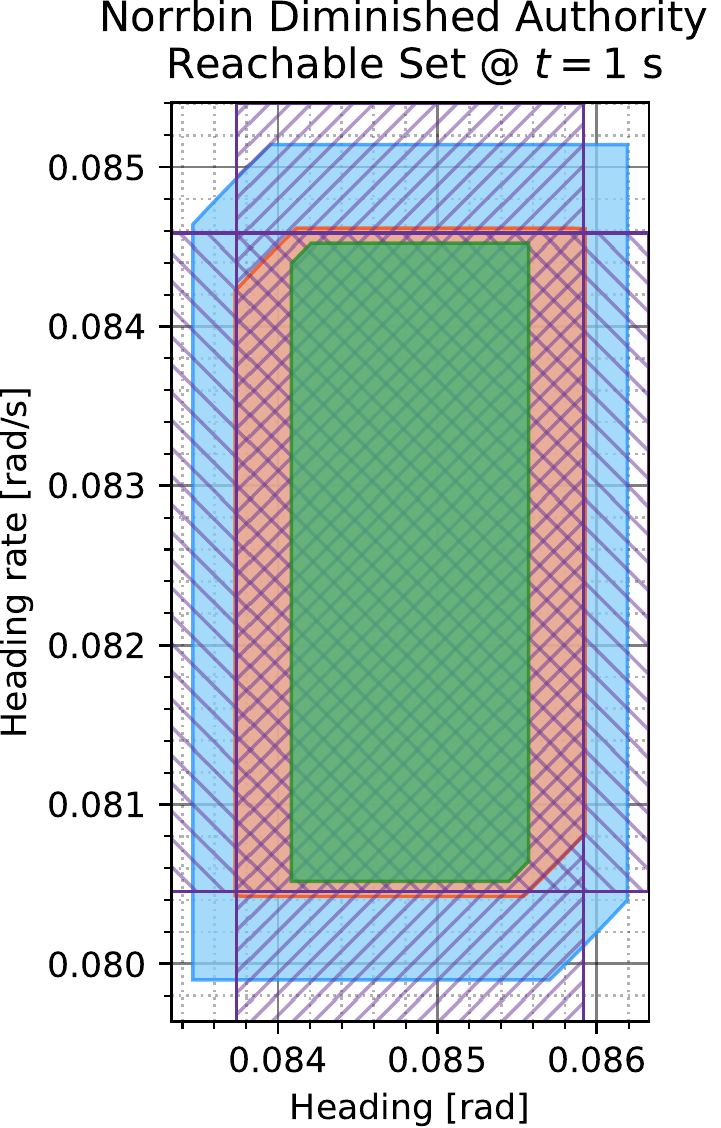}
    \caption{Reachable sets at 1 s}
    \label{fig:norrbin diminished 1000}
\end{subfigure}
\hfill
\begin{subfigure}[t]{0.5\linewidth}
    \centering
    \includegraphics[height=0.23\paperheight]{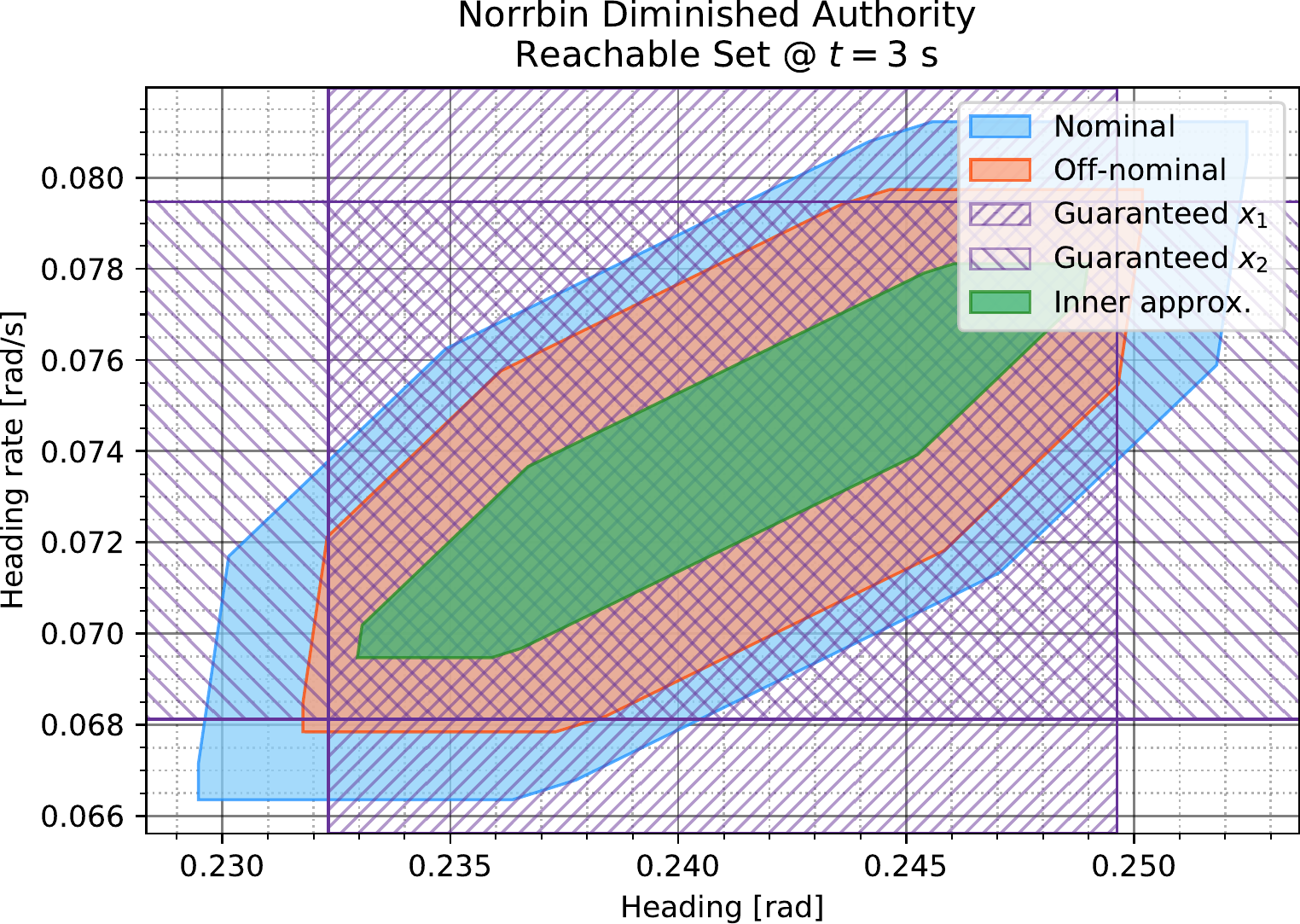}
    \caption{Reachable sets at 3 s}
    \label{fig:norrbin diminished 3000}
\end{subfigure}
\caption{Inner approximation of the off-nominal reachable set of the Norrbin model in the case of diminished control authority, as obtained using a ball-based slimming operation. The cross-hatched areas each denote an interval on the $i$-th Cartesian dimension in which it is guaranteed that there exists at least one state $y$ in the off-nominal reachable set such that $y_i = x_i$, where $x_i$ lies on the interval.}
\label{fig:norrbin diminished}
\end{figure*}

\subsubsection{Changed Dynamics}

In addition to the diminished control authority, we now also consider the following changed dynamics:
\begin{equation*}
    \dot{x}(t) = g(x(t), u(t)) = \begin{bmatrix}
        x_2 \\
        -\frac{v_{\text{s}}}{2 l} (x_2 + x_2^3)
    \end{bmatrix}
    +
    \begin{bmatrix}
        0 \\
        \frac{v_{\text{s}}^2}{2 l^2}
    \end{bmatrix}
    u,
\end{equation*}
where we consider $v_s < v$ and $v_s > v$, to capture a slowdown and speedup of the vessel, respectively.

\paragraph{Slowdown}
We first consider $v_{\text{s}} = 0.95 v = 4.75$ m/s. This slowdown causes the reachable set to shrink and shift slightly towards higher heading angles, since there is insufficient velocity to reach lower angles. The bound given in Lemma~\ref{lm:impaired dynamics norm bound} is used, where we define $\gamma(t)$  as follows:
\begin{equation*}
    \gamma(t) := \frac{|v_{\text{s}} - v|}{2 l} (m_2 + m_2^3) + \frac{v^2 - v_{\text{s}}^2}{2 l^2} \max_{u \in \mathscr{U}} |u|,
\end{equation*}
where $m_2 = \min_{y \in \mathbb{X}_{t}^\rightarrow (F, \mathscr{X}_0)} y_2$. Combining this bound with the trajectory deviation growth bound given at the beginning of this section, we obtain the conservative inner approximation shown in Fig.~\ref{fig:norrbin slowdown}. As can be clearly seen, not only is the off-nominal reachable set smaller, but it has also drifted to the top-left. This change is intuitively correct, since at a slower velocity rudder inputs become more effective as the vessel can make tighter turns at slower speeds. This phenomenon is reflected in the upward shift in heading rate and heading angle. A large area of the actual off-nominal set is lost due to the need for coping with

\begin{figure}[h]
\centerline{\includegraphics[height=0.23\paperheight]{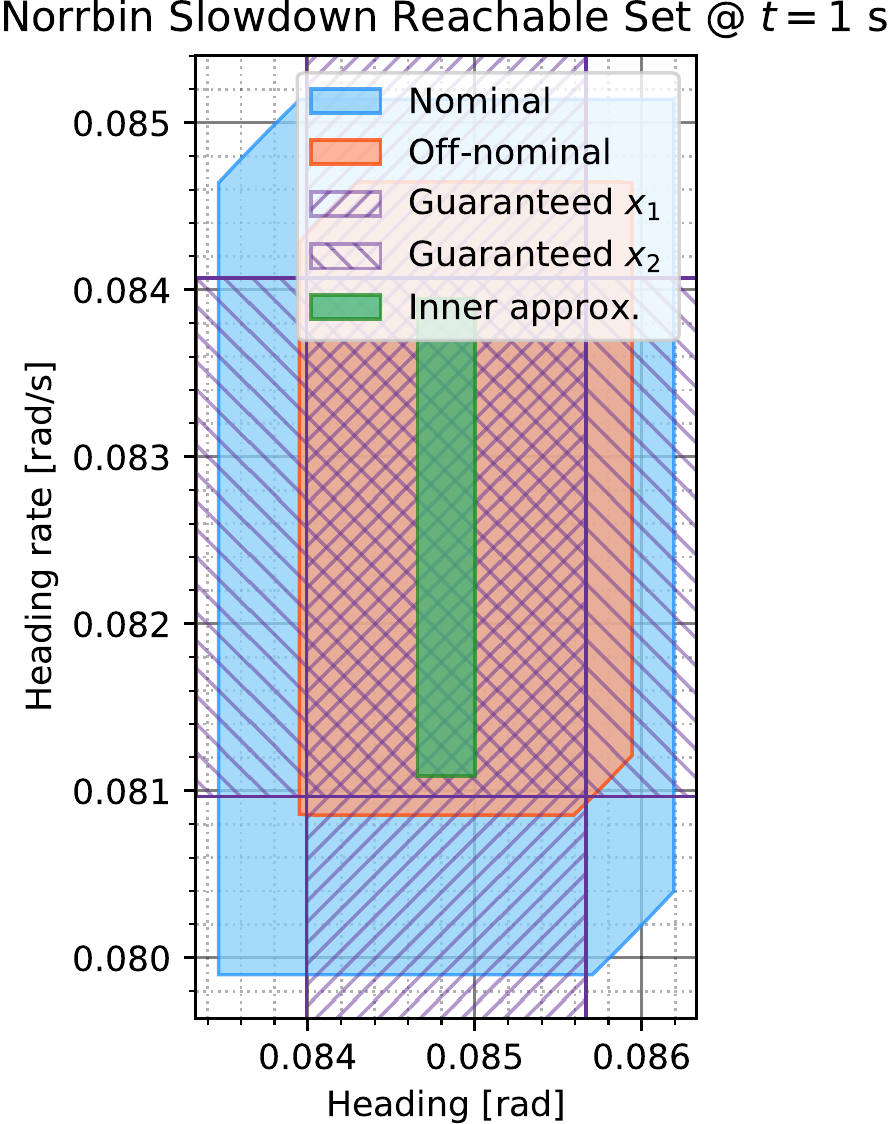}}
\caption{Inner approximation of the off-nominal reachable set of the Norrbin model in the case of decreased speed, as obtained using a conservative ball-based slimming operation.}
\label{fig:norrbin slowdown}
\end{figure}

\paragraph{Speedup}

We now demonstrate outer approximations in the case of changed dynamics. We still consider a diminishment in control authority, but in this case, $v_{\text{s}} = 1.4 v = 7$ m/s, indicating a speedup. In practice, one would like to know what the worst case trajectory could be in such a case, for example when attempting to avoid a high speed vessel. Instead of shrinking the nominal set, we now fatten it as shown in Corollary~\ref{cor:reachable set guaranteed overapproximation impaired general}. Our $\gamma(t)$ in this setting is as follows:
\begin{equation*}
    \gamma(t) := \frac{v_{\text{s}} - v}{2 l} (M_2 + M_2^3) + \frac{|v^2 - v_{\text{s}}^2|}{2 l^2} \max_{u \in \mathscr{U}} |u|.
\end{equation*}

With a similar trajectory deviation growth bound as in the case of a slowdown, we obtain an outer approximation as shown in Fig.~\ref{fig:norrbin speedup}, this time with an exact hyperrectangular fattening. In this case, it is also possible to apply a conservative ball-based fattening using the same radius as given previously.

In Fig.~\ref{fig:norrbin speedup}, it is clear that the off-nominal reachable set has shifted towards lower heading angles as rates, since the vessel will have less effective rudder authority at higher cruise speeds due to the its inertia. As a result, the outer approximation includes a large area of unused space towards the top right, since it needs to make up for both the translation and growth of the off-nominal reachable set with respect to the nominal reachable set.
\begin{figure}[h]
\centerline{\includegraphics[height=0.23\paperheight]{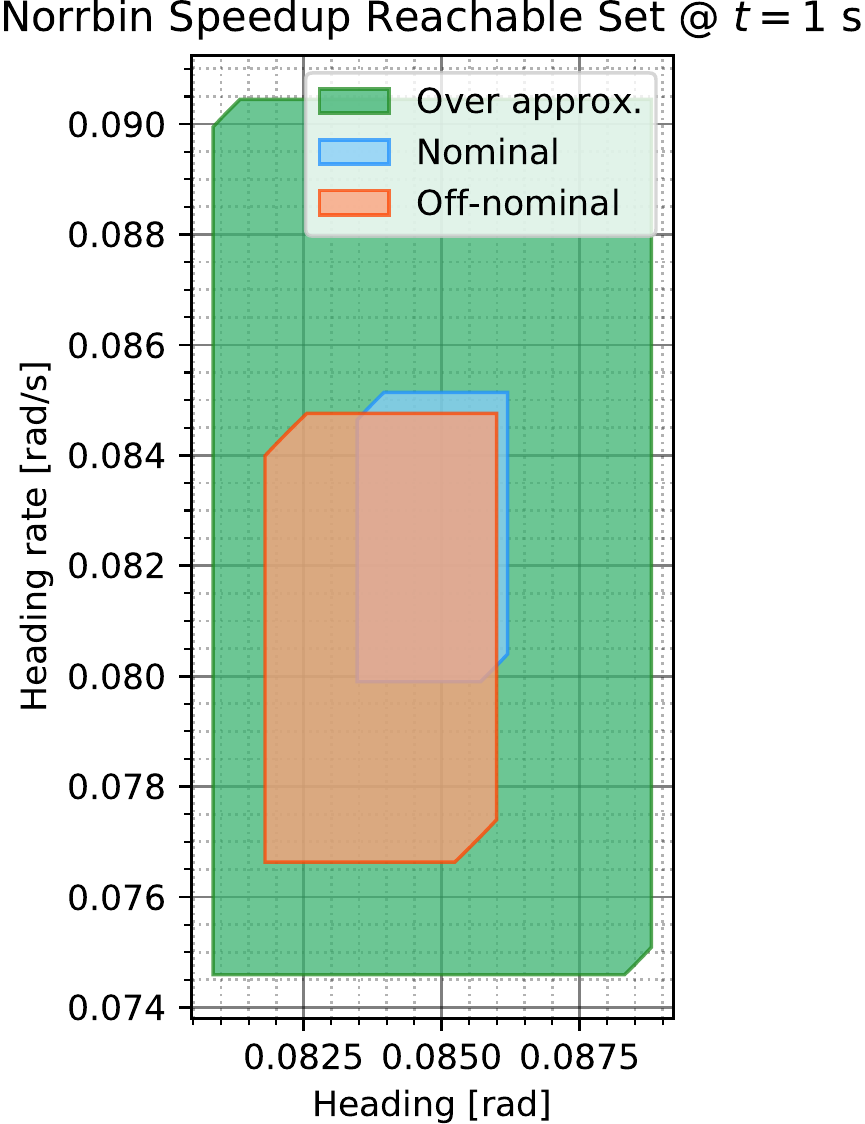}}
\caption{Outer approximation of the off-nominal reachable set of the Norrbin model in the case of increased speed, as obtained using a hyperrectangular fattening.}
\label{fig:norrbin speedup}
\end{figure}

\subsection{Cascaded System}\label{subsec:Scalable System Example}

To demonstrate that the approach given in Theorem~\ref{thm:reachable set guaranteed underapproximation impaired general} is scalable for high-dimensional systems, we present the following academic example.
We consider a lower triangular system; such systems often arise in practice when dealing with interconnected dynamical systems \cite{Zhang2013a}. Namely, we consider the system:
\begin{equation}\label{eq:academic example dynamics}
    \dot{x}(t) = \begin{bmatrix}
        \dot{x}_1 (t) & \cdots & \dot{x}_n (t)
    \end{bmatrix}\transp = A x(t) + B u(t) + d(t),
\end{equation}
where $x \in \mathbb{R}^n$ and $u \in \mathscr{U} \subseteq \mathbb{R}^m$, $A \in \mathbb{R}^{n \times n}$ is a lower triangular matrix, $B \in \mathbb{R}^{n \times m}$ is arbitrary, and $d : [t_0, \infty) \to \mathbb{R}^n$ is a differentiable function. The contribution of $d(t)$ is that of a nonlinear drift, possibly due to phenomena such as actuator bias or periodic disturbances.

We consider both the case of diminished control authority and changed dynamics below.

\subsubsection{Diminished Control Authority}

We consider an off-nominal set of admissible control inputs $\bar{\mathscr{U}} \subseteq \mathscr{U}$ such that $d_{\text{H}} (\mathscr{U}, \bar{\mathscr{U}}) \leq \epsilon$. Let $\dot{\tilde{x}}(t) = x(t) - \bar{x}(t) = A \tilde{x}(t) + B \tilde{u}(t)$, where $x(t)$ is a solution of the nominal system, and $\bar{x}(t)$ corresponds to the off-nominal system. It is then straightforward to show that a hyperrectangular trajectory deviation bound can be computed as follows:
\begin{equation}\label{eq:academic example hyperrectangular growth bound}
    |\dot{\tilde{x}}_i (t)| \leq \sum_{j = 1}^{i - 1} |a_{i,j}| \rho_j (t) + |a_{i,i}| |\tilde{x}_i (t)| + \sum_{j=1}^m |b_{i,j}| \epsilon,
\end{equation}
where each $\rho_j (t)$ is computed as per Lemma~\ref{lemma:hyperrectangular generalized boundedness of trajectory deviation}.
We will now show how to obtain the hyperrectangular trajectory deviation bound. Given the lower triangular structure of matrix $A$, by \eqref{eq:academic example hyperrectangular growth bound} we first compute $\rho_1 (t)$ from the following growth bound:
\begin{equation*}
    |\dot{\tilde{x}}_1 (t)| \leq |a_{1,1}| |\tilde{x}_1 (t)| + \sum_{j=1}^m |b_{1,j}| \epsilon.
\end{equation*}
By an application of Lemma~\ref{lm:impaired dynamics norm bound}, we can obtain $\rho_1 (t)$. Repeated application of \eqref{eq:academic example hyperrectangular growth bound} and Lemma~\ref{lm:impaired dynamics norm bound} will then yield the hyperrectangular deviation bound $\rho(t)$.


\subsubsection{Changed Dynamics}

Let us now consider the case where $d(t)$ in \eqref{eq:academic example dynamics} is replaced by $\bar{d}(t)$, such that $|d_i (t) - \bar{d}_i (t)| \leq \gamma_i (t)$. Let $\dot{\tilde{x}}(t) = x(t) - \bar{x}(t) = A \tilde{x}(t) + B \tilde{u}(t) + d(t) - \bar{d}(t)$, where $x(t)$ is a solution of the nominal system, and $\bar{x}(t)$ corresponds to the off-nominal system. It then suffices to modify \eqref{eq:academic example hyperrectangular growth bound} as follows:
\begin{equation}\label{eq:academic example hyperrectangular growth bound changed dynamics}
    |\dot{\tilde{x}}_i (t)| \leq \sum_{j = 1}^{i - 1} |a_{i,j}| \rho_j (t) + |a_{i,i}| |\tilde{x}_i (t)| + \sum_{j=1}^m |b_{i,j}| \epsilon + \gamma_i (t),
\end{equation}
which is obtained by the same arguments as in Lemma~\ref{lm:impaired dynamics norm bound}.

As an illustration of the bound given in \eqref{eq:academic example hyperrectangular growth bound}, we consider the following parameters:
\begin{equation*}
\begin{split}
    A_{i,j} &= \begin{cases}
        j/i, \quad &j \leq i \\
        0, \quad &j > i
    \end{cases}, \quad B_{i,j} = j/(9 - i), \quad d_i (t) = 0.1
\end{split}
\end{equation*}

We take the admissible set of control inputs to be $\mathscr{U} = \cart_{i=1}^n [-1, 1]$, and the diminished set of control inputs is $\bar{\mathscr{U}} = \cart_{i=1}^m [-0.9, 0.9]$, so that $\epsilon = d_{\mathrm{H}} (\mathscr{U}, \bar{\mathscr{U}}) = 0.1$. We take $n = 5$, and $m = 8$. By applying the bound on the trajectory deviation growth given in \eqref{eq:academic example hyperrectangular growth bound changed dynamics}, we can compute inner-approximations to the impaired reachable set as per Theorem~\ref{thm:reachable set guaranteed underapproximation impaired general}. We consider here a final time of $t = 0.25$, and an initial set of states $\mathscr{X}_0 = \{0\}$. The results are given as length fractions of the projections in Table~\ref{tab:length fractions lower triangular system}. Using the notation of Theorem~\ref{thm:reachable set guaranteed underapproximation impaired general}, for each $i$-th Cartesian dimension, the first column gives the ratio
\begin{equation*}
    \mathrm{length}[B_i \setminus ((\partial B)_i)_{+\rho_i(t)}] / \mathrm{length}[A_i],    
\end{equation*}
    and the second column shows
    \begin{equation*}
        \mathrm{length}[B_i \setminus ((\partial B)_i)_{+\Vert\rho(t)\Vert}] / \mathrm{length}[A_i].
    \end{equation*}

It can clearly be observed that the volume fractions of the inner approximations remain reasonably tight with increasing system dimension when considering the tightness of the hyperrectangular distance bound. These results serve to demonstrate that our method is capable of producing reasonably tight approximations even for systems with increased dimension by exploiting partially decoupled system structure. In comparison, ball-based slimming used in \cite{El-Kebir2021a} yields worse results, as can be seen by comparing the fourth and fifth column of Table~\ref{tab:length fractions lower triangular system}. For instance, a hyperrectangular shrinking operations defined by $\eta = [\begin{smallmatrix}0.1 & 10 \end{smallmatrix}]\transp$ yields a sufficient ball-based shrinking operation with radius $\Vert \eta \Vert \approx 10$, which would likely shrink away most of the first Cartesian dimension in practice. A similar phenomenon can be observed when considering dimensions 1 and 3 in Table~\ref{tab:length fractions lower triangular system}; using hyperrectangular bounds prevents excessive slimming of the reachable set.

\begin{table}[b]
\caption{Projected length ratios of the lower triangular system example by dimension using hyperrectangular and ball-based slimming operations.}
\label{tab:length fractions lower triangular system}
\begin{tabularx}{\linewidth}{l | X X}
    \toprule
    Dim. & Hyperrect. inner-approx./Off-nominal length & Ball-based inner-approx./Off-nominal length \\
    \midrule
    1 & 88.3\% & 29.8\% \\
    2 & 87.7\% & 41.7\% \\
    3 & 87.0\% & 52.8\% \\
    4 & 63.8\% & 18.8\% \\
    5 & 58.0\% & 30.3\% \\
    \bottomrule
\end{tabularx}
\end{table}

\subsubsection{Computational Complexity}

To show that the theory presented in this work is scalable on systems such as \eqref{eq:academic example dynamics}, we consider the computational complexity of a basic algorithmic implementation to compute $\rho_i$ for each $i$, as well as verifying if a state is guaranteed to lie in the off-nominal reachable set. Both of these tasks are subject to hard real-time constraints in practice, making it essential to study how their computational complexity grows.

\paragraph{Computing the Trajectory Deviation Bound}

We note that we must perform numerical integration to compute $G$, $\int a(\tau) \ \textrm{d}\tau$, and $\int b(\tau) \ \textrm{d}\tau$ in the Bihari inequality \eqref{eq:Bihari inequality}. To compute the inverse $G^{-1}$, one may use a root-finding scheme or an approximate look-up table (LUT) based approach. We consider here a LUT approach.

When using a method an explicit non-adaptive numerical integration scheme such as Euler's method or Runge-Kutta, it suffices to consider an a priori set integration step $h > 0$. Let us consider the reachable set in an interval $t \in [0, T]$, and take $h = T/N$ with $N \in \mathbb{N}$. By the results from \cite{Ruhela2014}, the computational complexity of a Runge-Kutta scheme is $\mathcal{O}(N)$. Since we will need to perform three rounds of numerical integration per dimension (for $G$, and in $a$ and $b$), which together take $\mathcal{O}(N)$. We store all value of $G$ and their argument $r$ in a lookup table of size $N$. Since values can be retrieved from an array in constant time, the complexity of numerical integration to populate the LUT combined with lookup is $\mathcal{O}(1) + \mathcal{O}(N) = \mathcal{O}(N)$. We then note that this process must be repeated for all $n$ dimensions, which gives computational complexity $\mathcal{O}(n N)$. Therefore, the value of $\rho(t)$, which is instrumental in producing guaranteed inner- and outer-approximations, can be computed in linear time with respect to the system dimension $n$.

\paragraph{Verifying Reachability of a State}

We now consider the complexity of verifying whether a state lies in the computed inner approximation of the reachable set. Let us assume that we have access to a \emph{signed distance function}, $\psi : \mathbb{R}^n \to \mathbb{R}$, of the nominal reachable set at time $t$ (see, e.g., \cite[p.~811]{Katopodes2019} for more information on signed distance functions). We assume that we can evaluate $\psi$ using $N_\psi$ primitive operations. Then, to evaluate whether or not a point $x \in \mathbb{R}^n$ lies in the inner approximation of the off-nominal reachable set, it suffices to check the following:
\begin{enumerate}
    \item Check if $\psi(x) \leq 0$; we must first check if $x$ lies in the nominal reachable set. If this is false, $x$ is not guaranteed to lie in the off-nominal reachable set.
    \item Check if $\psi(x) \leq - \min_i \rho_i(t)$; we must verify that $x$ lies at least distance $\min_i \rho_i(t)$ away from the boundary of the nominal reachable set. If this is false, $x$ is not guaranteed to lie in the off-nominal reachable set.
    \begin{enumerate}
        \item Check if $\psi(x) \leq -\Vert \rho(t) \Vert$. This verification is based on the ball-based slimming operation of \eqref{eq:ball-based slimming approximation}. If this is true, $x$ is guaranteed to lie in the off-nominal reachable set. If false, continue to the next step.
        \item Perform gradient ascent on $\psi$ starting at $x$, such that we reach $x'$ that satisfies $\psi(x') = 0$. This $x'$ is the point on the boundary of the nominal reachable set that is closest to $x$. Verify whether $\rho(t) \preceq |x - x'|$. If this inequality is true, $x$ is guaranteed to lie in the off-nominal reachable set.
    \end{enumerate}
\end{enumerate}

In the above algorithm, it will take at least one evaluation of $\psi$ to verify whether $x$ is guaranteed to be in the off-nominal reachable set. Doing so requires $N_\psi$ operations, and corresponds to step 1). An evaluation of $\rho(t)$ will cost $\mathcal{O}(n)$ operations as discussed previously, which yields a complexity of $\mathcal{O}(n N_\psi)$. Evaluating the norm of $\rho(t)$ can be done in linear time as part of step 2a), but performing gradient ascent in step 2b) may require a significant number of evaluations of $\psi$. It is possible to truncate the gradient ascent algorithm based on a maximum number of evaluations of $\psi$, say $N_{\mathrm{eval}}$. Given some $x'' \in \mathbb{R}^n$ obtained after $N_{\mathrm{eval}} - 1$ evaluations of $\psi$, it is clear that $x' \in \{x''\}_{+|\psi(x'')|}$. We can check if for each $i = 1, \ldots, n$, it holds that $|x_i - x_i''| + |\psi(x'')| \leq \rho_i (t)$. If this inequality is true, then $x$ is guaranteed to lie in the off-nominal reachable set, and if not, then $x$ cannot be verified with certainty. Therefore, it is possible to verify guaranteed reachability with complexity $\mathcal{O}(n)$.
%
%

\subsection{Interconnected System}

To demonstrate the results shown in Subsection~\ref{subsec:Scalable System Example} on a different system structure, we consider a cascaded system of linear equations \cite{Saif1992}. Let there be $N \in \mathbb{N}$ interconnected systems, such that the $i$-th subsystem may only depend on its own states and inputs, as well as the states of the previous subsystem ($i-1$). The overall system thus takes the form:
\begin{equation}
    \dot{x}(t) = (A + K) x(t) + B u(t),
\end{equation}
where
\begin{equation*}
\begin{split}
    A &= \mathrm{diag}(\{ A^{(1)}, A^{(2)}, \ldots, A^{(N)}\}) \in \mathbb{R}^{(\Sigma_{i=1}^N n_i) \times (\Sigma_{i=1}^N n_i)}, \\
    B &= \mathrm{diag}(\{ B^{(1)}, B^{(2)}, \ldots, B^{(N)}\}) \in \mathbb{R}^{(\Sigma_{i=1}^N n_i) \times (\Sigma_{i=1}^N m_i)}, \\
    K &= \left[
        \begin{Blockmatrix}
            \block{1}{1}{1}{1}{0}
            \block{2}{1}{2}{1}{K^{(2)}}
            \block[]{3}{2}{3}{2}{\ddots}
            \block{4}{3}{4}{3}{K^{(N)}}
            \block{4}{4}{4}{4}{0}
        \end{Blockmatrix}
    \right] \in \mathbb{R}^{(\Sigma_{i=1}^N n_i) \times (\Sigma_{i=1}^N n_i)}.
\end{split}
\end{equation*}

In the above definition, we have $K^{(i)} \in \mathbb{R}^{n_i \times n_{i - 1}}$ for all $i = 2, \ldots, n$. For the first system, we can compute the hyperrectangular deviation bound as in \cite{El-Kebir2021a}, simply by considering the ball-based growth bound for that system. Doing so yields
\begin{equation}\label{eq:first system growth rate bound interconnected example}
    \Vert \dot{\tilde{x}}^{(1)} (t) \Vert \leq \sum_{i=1}^{n_1} \sum_{j=1}^{n_1} |A_{i,j}| \Vert \tilde{x}^{(1)}(t) \Vert + |B_{i,j}| \epsilon.
\end{equation}
From this inequality, we can compute the trajectory deviation bound $\rho(t)$ as per Corollary~\ref{corollary:generalized boundedness of trajectory deviation}. For subsequent systems, we find
\begin{equation}\label{eq:subsequent system growth rate bound interconnected example}
\begin{split}
    \Vert \dot{\tilde{x}}^{(k)} (t) \Vert &\leq \sum_{i=1}^{n_k} \sum_{j=1}^{n_{k-1}} |K_{i,j}^{(k)}| \rho^{(k-1)} (t) \\
    &+ \sum_{i=1}^{n_k} \sum_{j=1}^{n_k} |A_{i,j}| \Vert \tilde{x}^{(k)} (t) \Vert + |B_{i,j}| \epsilon,
\end{split}
\end{equation}
for $k > 1$. In case any of the constituent systems possess a decoupled structure, simplifications of the form of \eqref{eq:academic example hyperrectangular growth bound} can be made in \eqref{eq:first system growth rate bound interconnected example} or \eqref{eq:subsequent system growth rate bound interconnected example}.

\subsubsection{Numerical Example}

We consider the following system:
\begin{equation}
\begin{split}
    \dot{x}(t) &= \begin{bmatrix}
        -1 & 1 & 0 & 0 & 0 \\
        0 & -1 & 0 & 0 & 0 \\
        0 & 0 & -0.5 & 0.5 & -0.1 \\
        0 & 0 & -0.5 & -0.1 & 1 \\
        0 & 0 & 0 & 0.1 & -0.5
    \end{bmatrix} x(t) \\
    &+
    \begin{bmatrix}
        0 & 0 & 0 & 0 & 0 & 0 \\
        0 & 0 & 0 & 0 & 0 & 0 \\
        0 & 0 & 0 & 0 & 0 & 0 \\
        0.1 & 0 & 0 & 0 & 0 & 0 \\
        0 & 0.5 & 0 & 0 & 0 & 0 \\
    \end{bmatrix} x(t) +
    \begin{bmatrix}
        0.1 & 0 \\
        1 & 0 \\
        0 & 0.1 \\
        0 & 0.1 \\
        0 & 0.1
    \end{bmatrix} u(t),
\end{split}
\end{equation}
where the set of admissible control inputs is $\mathscr{U} = [-2, 2]^2$, the set of initial states is $\mathscr{X}_0 = \{0\}$. We take the impaired set of admissible control inputs to be $\bar{\mathscr{U}} = [-1.9, 1.9]^2$, such that $\epsilon = d_{\mathrm{H}} (\mathscr{U}, \bar{\mathscr{U}}) = 0.1$. Using the approach of Theorem~\ref{thm:reachable set guaranteed underapproximation impaired general}, we obtain the results as shown in Table~\ref{tab:length fractions interconnected system}.

\begin{table}[t]
\caption{Projected length ratios of the interconnected system example by dimension using hyperrectangular and ball-based slimming operations.}
\label{tab:length fractions interconnected system}
\begin{tabularx}{\linewidth}{l | X X}
    \toprule
    Dim. & Hyperrect. inner-approx./Off-nominal length & Ball-based inner-approx./Off-nominal length \\
    \midrule
    1 & 61.0\% & 34.3\% \\
    2 & 95.5\% & 89.7\% \\
    3 & 67.4\% & 0\% \\
    4 & 71.7\% & 0\% \\
    5 & 80.4\% & 13.6\% \\
    \bottomrule
\end{tabularx}
\end{table}

As can be observed in Table~\ref{tab:length fractions interconnected system}, the inner-approximations based on the hyperrectangular slimming outperform those based on ball-based slimming operations. In particular, in states 3 and 4, the ball-based slimming operations eliminate the entire reachable set, which is not the case with hyperrectangular slimming operations. The computations applied to the system in this example are scalable while preserving relatively tight bounds, provided that the system structure permits decoupling of subsystems as shown here.

\section{Conclusion}
\label{sec:conclusions}

In this work, we have introduced a new technique for efficiently computing both inner and outer approximations to a reachable set in case of changed dynamics and diminished control authority, given basic knowledge of the trajectory deviation growth as well as a precomputed nominal reachable set. This work expands on previous work by extending the theory to changes in dynamics, and lifting the assumption of convexity of the reachable sets. To obtain an inner approximation of the reachable set under diminished control authority, we have given an integral inequality that provides an upper bound on the minimal trajectory deviation between the nominal and off-nominal systems. We have extended the classical norm bound on the trajectory deviation to a hyperrectangular bound, allowing us to compute both inner and outer approximations of the off-nominal reachable set based on the nominal set, regardless of the convexity of the reachable set. Similarly to our previous results, these results can be applied online on systems at a low computational cost.

We have demonstrated our approach by three examples: a model of the heading dynamics of a vessel, a lower triangular system, and an interconnected linear system. In general, the use of a hyperrectangular growth bound is superior to a norm bound for systems that have one or more integrators. The numerical examples indicate that the use of hyperrectangular slimming operations would produce tighter inner approximations, coupled with periodic reinitialization of the reachable set. As was mentioned in previous work, the tightness of both the inner and outer approximation are strongly related to the quality of the trajectory deviation bound, as well as any additional drift that appears as part of a change in dynamics. We have shown that the ability to compute these approximations online can have practical application to control of dynamical systems in off-nominal conditions. This was shown in the second example, where the computational complexity was shown to be linear in the system dimension for a lower triangular system. Finally, in the third example, it was shown how system structure can be leveraged when dealing with interconnected systems in the context of formulating an efficient hyperrectangular growth bound that consists of several coupled ball-based growth bounds. The latter approach was shown to be applicable to larger systems, provided that it is possible to decouple some subsystems from each other.

In future work, we aim to study the utility of a bounding method based on non-axis-aligned hyperrectangles, as could be described by zonotopes, insofar as obtaining tighter growth bounds and approximations is concerned. A potential avenue for this work would lie in considering principle components of the system using singular value decomposition \cite{Amsallem2012}, or by considering the system structure itself (e.g., when the set of velocities of a system lies in a subspace). In the same direction, (normalizing) state-space transformations may also prove to be useful in obtaining tighter approximations by easing magnitude difference between states. In addition, generalized slimming and fattening operations that are based on sets that are not centered at the origin may also prove to be key to obtaining tighter approximations in the case of changes in dynamics. Finally, real-time applications of the theory presented here will be studied in future work, with a focus on safety-critical predictive control.


\appendices

\section{Generalizations to the Theory}\label{app:generalizations}

In the theory presented in Section~\ref{sec:main}, a number of assumptions can be weakened to address a larger class of dynamical systems; we present these relaxations below.

For the result of Lemma~\ref{lemma:path-connectedness of solution set}, it suffices that the multifunction $F : [0, \infty) \times \mathbb{R}^n \rightrightarrows \mathbb{R}^n$ satisfies the following properties \cite[Thm.~1, p.~1010]{Staicu2004}:
    \begin{enumerate}
        \item $F$ is $\mathscr{L} \otimes \mathscr{B}(\mathbb{R}^n)$-measurable, as defined in \cite[p.~1007]{Staicu2004};
        \item $F$ is Lipschitz with respect to $x$, i.e., there exists $l \in L^1_{\mathrm{loc}}([0, \infty), \mathbb{R})$, such that $l(t) > 0$, and for any $x, y \in \mathbb{R}^n$, it holds that
            $d_{\mathrm{H}} (F(t, x), F(t, y)) \leq l(t) \Vert x - y \Vert$
            for a.e. $t \in [0, \infty)$;
        \item There exists $\beta \in L^1_{\mathrm{loc}}([0, \infty), \mathbb{R})$ such that $d_{\mathrm{H}} (\{0\}, F(t, 0)) \leq \beta(t)$
            for a.e. $t \in [0, \infty)$.
    \end{enumerate}
If $F(t, x)$ is continuous, Lipschitz in $x$, and has closed, path-connected values, as assumed in the main part of the paper, it satisfies these assumptions. Namely, 1) is satisfied by continuity of $F$, while 2) and 3) are satisfied by the Lipschitz condition on $F(t, x)$ in $t$ and $x$.

For the claim of Proposition~\ref{prop:solution set values continuum}, it suffices that in addition to assumptions 1)--3) above, multifunction $F$ possesses the \emph{Scorza--Dragoni property} \cite[Def.~19.12, p.~91]{Gorniewicz1999}:

\begin{definition}
    A multifunction $F : [0, \infty) \times \mathbb{R}^n \rightrightarrows \mathbb{R}^n$ with closed values is said to have the \emph{Scorza--Dragoni property} if, for all $\delta > 0$, there exists a closed subset $A_\delta \subset [0, \infty)$ such that $\mu([0, \infty) \setminus A_\delta) \leq \delta$, where $\mu$ is the Lebesgue measure, and $F$ is continuous on $A_\delta \times \mathbb{R}^n$.
\end{definition}

It trivially follows that $F$ satisfies the Scorza--Dragoni property if it is continuous and has closed values.

Finally, for Lemma~\ref{lemma:connectedness of reachable set} to hold, conditions 1)--3) above and the Scorza--Dragoni property again form sufficient conditions; in its proof, the solution set $S_F$ is indeed continuous if conditions 1)--2) are met, by Corollary 4.5 in \cite{Zhu1991}.

\bibliographystyle{IEEEtran}
\bibliography{tacarxiv}

\begin{IEEEbiography}[{\includegraphics[width=1in,height=1.25in,clip,keepaspectratio]{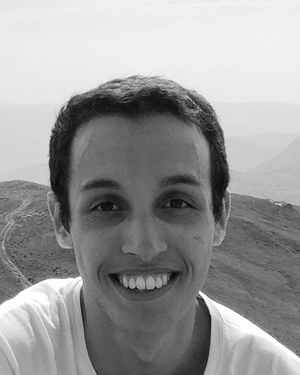}}]{Hamza El-Kebir} received the B.S. degree in
aerospace engineering from Delft University of Technology, Delft, The Netherlands, in 2020. He is currently
working toward the Ph.D. degree from the Department of Aerospace Engineering, University of Illinois Urbana-Champaign, Urbana, IL, USA. His current research interests are in safe control and estimation of systems experiencing uncertain failure modes, changes in dynamics, and degradation of control authority. He also focuses on infinite-dimensional control of distributed parameter systems for safe autonomous surgery.
\end{IEEEbiography}

\begin{IEEEbiography}[{\includegraphics[width=1in,height=1.25in,clip,keepaspectratio]{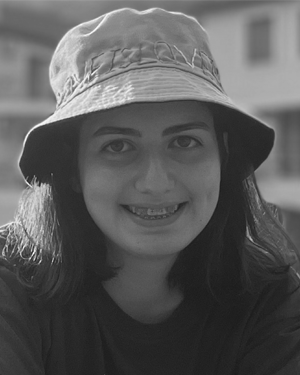}}]{Ani Pirosmanishvili} is currently working toward her B.S. degree in aerospace engineering at the Department of Aerospace Engineering at the University of Illinois Urbana-Champaign, Urbana, IL, USA. Her current research interests include estimation of the performance of autonomous systems after degradation, with applications on high-fidelity dynamic models. 
\end{IEEEbiography}

\begin{IEEEbiography}[{\includegraphics[width=1in,height=1.25in,clip,keepaspectratio]{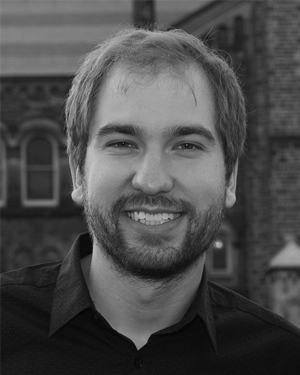}}]{Melkior Ornik} received the Ph.D. degree from the University of Toronto, Toronto, ON, Canada, in 2017. He is currently an Assistant Professor with the Department of Aerospace Engineering and the Coordinated Science Laboratory, University of Illinois Urbana-Champaign, Urbana, IL, USA. His research focuses on developing theory and algorithms for learning and planning of autonomous systems operating in uncertain, complex, and changing environments, as well as in scenarios where only limited knowledge of the system is available.
\end{IEEEbiography}

\end{document}